\documentclass[11pt,a4paper]{amsart}
\usepackage{amssymb,amsmath,epsfig,graphics,mathrsfs,enumerate,verbatim}
\usepackage[pagebackref,colorlinks=true,linkcolor=blue,citecolor=blue]{hyperref}
\usepackage{fancyhdr}
\usepackage{hyperref}
\hypersetup{
 colorlinks   = true,
 urlcolor     = blue,
 linkcolor    = blue,
 citecolor   = red ,
 bookmarksopen=true
}

\usepackage{amsmath}
\usepackage{amsfonts}
\usepackage{amssymb}
\usepackage{amsthm}
\usepackage{epsfig,graphics,mathrsfs}
\usepackage{graphicx}
\usepackage{dsfont}

\usepackage[usenames, dvipsnames]{color}

\usepackage{hyperref}

\newcommand*\MSC[1][1991]{\par\leavevmode\hbox{%
\textit{#1 Mathematical subject classification:\ }}}
\newcommand\blfootnote[1]{%
  \begingroup
  \renewcommand\thefootnote{}\footnote{#1}%
  \addtocounter{footnote}{-1}%
  \endgroup
}

\def \phi {\varphi}

\def \R {\mathbb{R}}

\def \G{\Gamma}

\def \vf{\varphi}

\newcommand{\Rn}{\mathbb R^n}

\newcommand{\Om}{\Omega}
\newcommand{\Hn}{\mathbb H^n}

\newcommand{\p}{\partial}

\newcommand{\bG}{\mathbb {G}}
\newcommand{\bg}{\mathfrak g}

\newcommand{\la}{\lambda}

\numberwithin{equation}{section}

\newcommand{\beq}{\begin{equation}}
\newcommand{\bea}[1]{\begin{array}{#1} }
\newcommand{\eeq}{ \end{equation}}
\newcommand{\ea}{ \end{array}}

\newcommand{\ds}{\displaystyle}

\newcommand{\ve}{\varepsilon}

\newcommand{\nh}{\nabla_H}
\newcommand{\sul}{\Delta_H}

\newcommand{\sa}{\langle}
\newcommand{\da}{\rangle}

\newcommand{\QK}{\boldsymbol {G\,(\mathbb{K})}}
\newcommand{\algg}{\mathfrak g}
\newcommand{\algn}{\mathfrak n}

\newcommand{\Kn}{\mathbb{K}^n}
\newcommand{\C}{\mathbb{C}}
\newcommand{\Quat}{\mathbb{H}}

\newcommand{\Vone}{\mathfrak h}
\newcommand{\Vtwo}{ \mathfrak v}

\newtheorem{theorem}{Theorem}[section]
\newtheorem{lemma}[theorem]{Lemma}
\newtheorem{proposition}[theorem]{Proposition}

\newtheorem{remark}[theorem]{Remark}
\newtheorem{definition}[theorem]{Definition}
\newtheorem{example}[theorem]{Example}

\numberwithin{equation}{section}

\begin{document}

\title[Overdetermined problems in groups,  etc.]{Overdetermined problems in groups of Heisenberg type: conjectures and partial results}

\blfootnote{\MSC[2020]{35N25, 35H20, 31C15}}

\date{}

\begin{abstract}
In this paper we formulate some conjectures in sub-Riemannian geometry concerning a characterisation of the Koranyi-Kaplan ball in a group of Heisenberg type through the existence of a solution to suitably overdetermined problems. We prove an integral identity that provides a rigidity constraint for one of the two problems. By exploiting some new invariances of these Lie groups, for domains having partial symmetry we solve these problems by converting them to known results for the classical $p$-Laplacian.
\end{abstract}

\author{Nicola Garofalo}

\address{Dipartimento d'Ingegneria Civile e Ambientale (DICEA)\\ Universit\`a di Padova\\ Via Marzolo, 9 - 35131 Padova,  Italy}
\vskip 0.2in
\email{nicola.garofalo@unipd.it}

\thanks{The first author has been supported in part by a Progetto SID (Investimento Strategico di Dipartimento): ``Aspects of nonlocal operators via fine properties of heat kernels", University of Padova, 2022. He has also been partially supported by a Visiting Professorship at the Arizona State University.}

\author{Dimiter Vassilev}

\address{University of New Mexico\\
Department of Mathematics and Statistics\\
311 Terrace Street NE\\
Albuquerque, NM  87106}
\vskip 0.2in
\email{vassilev@unm.edu}

\maketitle

\tableofcontents

\section{Introduction}\label{S:intro}

In 1970 R. Fosdick (now Professor Emeritus of Aerospace Engineering at the University of Minnesota) asked the following question: suppose that in a smooth domain $\Om\subset \Rn$ one has a solution to the Dirichlet problem $\Delta f = - 1$, such that $f = 0$ on $\p \Om$. Is it true that $f$ satisfies the overdetermined condition $\frac{\p f}{\p \nu} = c$ on $\p \Om$ if and only if $\Om$ is a ball $B(x_0,R)$ and
\[
f(x) = \frac{R^2-|x-x_0|^2}{2n}\ ?
\]
As it is well-known, in his celebrated paper \cite{se} J. Serrin provided a positive answer to Fosdick's question, even for a more general class of elliptic equations, by introducing in pde's what is nowadays known as the Alexandrov-Serrin method of moving planes.  In \cite{hans} H. Weinberger presented  an alternative approach to Serrin's result based on integral identities and the strong maximum principle. Both these works have been deeply influential and have generated an enormous amount of
research.

In this paper we formulate two conjectures, and present some progress toward their solution, about the characterisation via an overdetermined boundary value problem of a class of geometric objects which play the role of the Euclidean balls in certain nilpotent Lie groups. This latter statement should be taken with a proviso since, unlike the Euclidean balls, the sets of interest in this paper are not perimeter minimising, nor they have constant mean curvature. Nonetheless, they play an important role in analysis since they are the level sets of the fundamental solutions of:
\begin{itemize}
\item[(i)] The Euler-Lagrange equations of the natural $p$-energies, $1<p<\infty$ (see Theorems A and B below);
\item[(ii)] A class of pseudodifferential operators that represent the fractional powers of the conformal horizontal Laplacian in CR geometry (see Theorem C below).
\end{itemize}

The specific framework of this paper is that of a group of Heisenberg type $(\bG,\circ)$ with logarithmic cordinates $x = (z,\sigma)$, and the central character will be the Kor\'anyi gauge function
\begin{equation}\label{gauge}
N(x) = (|z|^4 + 16 |\sigma|^2)^{1/4},
\end{equation}
with its level sets, the \emph{gauge balls} $B_R(x_0) = \{x\in \bG\mid N(x_0^{-1}\circ x)<R\}$. When the center is the identity $e\in \bG$, we simply write $B_R$, instead of $B_R(e)$. The relevance of \eqref{gauge} goes back to the following remarkable 1973 discovery of G. Folland in \cite[Theorem 2]{Fo} which played a critical role in the seminal work \cite{FScpam}.

\medskip

\noindent \textbf{Theorem A.}\
\emph{The fundamental solution with pole at the group identity of the horizontal Laplacian $-\sul$ in the Heisenberg group $\Hn$ is given by
\begin{equation}\label{folland}
\mathscr E(x) = \frac{C(n)}{N(x)^{Q-2}},
\end{equation}
where $C(n)>0$ is a suitable explicit constant (see \eqref{C} below, in which one has to take $s= k= 1$ and $m=2n$), and $Q = 2n+2$ is the homogeneous dimension of $\Hn$}.

\medskip

Theorem A was generalised by A. Kaplan in \cite[Theorem 2]{Ka} to all Lie groups of Heisenberg type (see also the earlier work \cite[Theorem 1]{KP}, where the same result was proved for groups of Iwasawa type). To provide further instances of the role of \eqref{gauge}, and to state our results, we next introduce the class of relevant Lie groups. Consider a simply connected, stratified nilpotent Lie group $\bG$ of step two. This means that, if we denote its Lie algebra by $\bg = \mathfrak h\oplus \mathfrak v$, where $\mathfrak h$ is the bracket-generating layer and $\mathfrak v$ is the center, then we have $[\mathfrak h,\mathfrak h] = \mathfrak v$ and $[\mathfrak h,\mathfrak v] = \{0\}$. We let $m = \operatorname{dim} \mathfrak h$, $k = \operatorname{dim} \mathfrak v$, and indicate by
$\{e_1,...,e_m\}$,\ $\{\ve_1,....,\ve_k\}$,
fixed orthonormal basis of respectively $\mathfrak h$ and $\mathfrak v$.
As above, we continue to denote with $x = (z,\sigma)$, $z\in \mathfrak h, \sigma\in \mathfrak v$, the logarithmic coordinates of a point $x\in \bG$. Throughout this paper, we indicate with $Q = m + 2k$ the homogeneous dimension associated with the anisotropic group dilations $\delta_\la x = (\la z,\la^2 \sigma)$.

Consider the linear mapping $J:\mathfrak v \to$ End$(\mathfrak h)$, defined by
\begin{equation}\label{kap}
\sa J(\sigma)z,z'\da = \sa[z,z'],\sigma\da.
\end{equation}
By anti-symmetry of the bracket, it is clear that $J(\sigma)^\star = - J(\sigma)$, and that $\sa J(\sigma)z,z\da = 0$. The Baker-Campbell-Hausdorff formula (see \eqref{BCH} below) and \eqref{kap} prescribe the non-Abelian multiplication in $\bG$
\begin{align}\label{grouplaw}
& x \circ x'  = \big(z +z',\sigma + \sigma' + \frac 12 \sum_{\ell=1}^{k} \sa J(\ve_\ell)z,z'\da\ve_\ell\big).
\end{align}
If $L_x(x') = x \circ x'$ is the operator of left-translation in $\bG$, and $d L_x$ denotes its differential, we define left-invariant vector fields in $\bG$ by setting
\begin{equation}\label{vfields}
X_i = d L_x(e_i),\ i=1,...,m,\ \ \ \ T_\ell = d L_x(\ve_\ell),\ \ell=1,...,k.
\end{equation}
Given a function $f$, we denote by $|\nh f|^2 = \sum_{i=1}^m (X_i f)^2$ the left-invariant horizontal \emph{carr\'e du champ} associated with $\{e_1,...,e_m\}$, and consider the Dirichlet $p$-energy
\begin{equation}\label{pen}
\mathscr E_p(f) = \frac 1p \int |\nh f|^p dx,\ \ \ \ \ \ 1<p<\infty.
\end{equation}
When $p=2$, minimisers of \eqref{pen} are harmonic functions, i.e., solutions of the partial differential equation $\sul f = \sum_{i=1}^m X_i^2 f =0$. Using \eqref{grouplaw} and the definition of the vector fields $X_i$, one obtains
\begin{equation}\label{sulla}
\sul f= \Delta_z f + \frac 14 \sum_{\ell,\ell' = 1}^k \sa J(\ve_\ell)z,J(\ve_{\ell'})z\da \p_{\sigma_\ell}\p_{\sigma_{\ell'}} f + \sum_{\ell = 1}^k \p_{\sigma_\ell} \Theta_\ell f,
\end{equation}
where we have denoted $\Theta_\ell = \sum_{s=1}^m \langle J(\ve_\ell)z,e_s\rangle \p_{z_s}$.
The operator $\sul$ fails to be elliptic at every point $x\in \bG$, but by a famous result in \cite{Ho} it is hypoelliptic.

A stratified nilpotent Lie group $\bG$ of step two is called of \emph{Heisenberg type} if for every $\sigma\in \mathfrak v$ such that $|\sigma| = 1$, the mapping $J(\sigma)$ is orthogonal on $\mathfrak h$, see \cite{Ka}. This is equivalent to saying that for every $\sigma\in \mathfrak v$ one has
\begin{equation}\label{J2}
J(\sigma)^2 = - |\sigma|^2\  I_{\mathfrak h},
\end{equation}
thus in particular $J$ induces a complex structure in $\bG$.
The prototype of this class of Lie groups is the Heisenberg group $\Hn$, in which there exists one single map $J : \mathfrak h\to \mathfrak h$, given by $J z = z^\perp$. Besides $\Hn$,
there exists in nature a very rich supply of such Lie groups. In  the Iwasawa decomposition of a simple Lie group of rank one,   the nilpotent component  is a group of Heisenberg type, see \cite{Ka2}, \cite{Adam85} and \cite{CDKR}. We remark that \eqref{J2} introduces important additional symmetries in $\bG$. For instance, it gives
\begin{equation}\label{ort}
|J(\sigma)z|^2 = \sa J(\sigma)z,J(\sigma)z\da = |\sigma|^2 |z|^2,
\end{equation}
and by polarization of \eqref{kap} and \eqref{ort} we obtain
\begin{equation}\label{ip}
\sa J(\ve_\ell)z,J(\ve_{\ell'})z\da = |z|^2 \delta_{\ell \ell'},\ \ \ \ \ \ \ \ \ \ell, \ell' = 1,...,k.
\end{equation}
As a consequence of \eqref{ip}, the horizontal Laplacian \eqref{sulla} becomes
\begin{equation}\label{sullaH}
\sul f = \sum_{i=1}^m X_i^2 f = \Delta_z f + \frac{|z|^2}4 \Delta_\sigma f + \sum_{\ell = 1}^k \p_{\sigma_\ell} \Theta_\ell f.
\end{equation}

When $p\not=2$ in \eqref{pen}, minimisers are weak solutions of the following quasilinear Euler-Lagrange equation
\begin{equation}\label{pH}
\Delta_{H,p} f = \operatorname{div}_H (|\nh f|^{p-2} \nh f) = \sum_{i=1}^m X_i(|\nh f|^{p-2} X_i f) = 0.
\end{equation}
Degenerate nonlinear equations of this type arise, e.g., in the foundational work of Kor\'anyi and H. M. Reimann \cite{KoR}, \cite{KoRaim}, and of Mostow \cite{Mo} and Mostow and Margulis \cite{MM}, see also \cite{CDG}, \cite{D}, \cite{CDGcap} and \cite{HH}.

As the following result shows, the gauge \eqref{gauge} is also connected to the nonlinear equation \eqref{pH}, see \cite[Theorem 2.1]{CDGcap} (the case $p=Q$ was also independently discovered in \cite{HH}).

\medskip

\noindent \textbf{Theorem B.}\
\emph{For any $1<p<\infty$ there exists an explicit universal constant $C_p>0$ such that
the fundamental solution of the nonlinear operator $\Delta_{H,p}$ (with singularity in $e\in \bG$) is given by
\begin{equation}\label{fs}
\Gamma_p(x) = \begin{cases}
\frac{C_p}{N(x)^{\frac{Q-p}{p-1}}},\ \ \ \ \text{when}\ p\not=Q,
\\
\\
C_p \log N(x),\ \ \ \ \ \text{when}\ p=Q.
\end{cases}
\end{equation}}
We mention that, inspired by \eqref{fs}, it was proved in \cite{CDGcap} that the fundamental solutions $\G_p(x,x_0)$ of general classes of nonlinear equations modelled on \eqref{pH} behave near their singularity $x_0$ as follows:
\begin{equation*}\label{psingsol}
C\bigg(\frac{d(x,x_0)^p}{\operatorname{Vol}(B(x_0,d(x,x_0)))}\bigg)^{\frac{1}{p-1}}
\leq
\G_p(x,x_0)
\leq C^{-1}
\bigg(\frac{d(x,x_0)^p}{\operatorname{Vol}(B(x_0,d(x,x_0)))}\bigg)^{\frac{1}{p-1}}.
\end{equation*}
This estimate generalises a celebrated result of Nagel, Stein and Wainger in the linear case $p = 2$, see \cite{NSW}.

\medskip

A third example attesting to the importance of the function \eqref{gauge} originates from conformal geometry. In the paper \cite{BFM} Branson, Fontana and Morpurgo introduced in $\Hn$, for $0<s\le 1$, a  pseudodifferential operator $\mathscr L_{s}$ which is the counterpart of the fractional powers of the  conformal Laplacian from Riemannian geometry. In a group of Heisenberg type $\bG$ such nonlocal operators $\mathscr L_{s}$ are defined by the spectral formula
\begin{equation}\label{bfmH}
\mathscr L_s = 2^s (-\Delta_\sigma)^{s/2} \frac{\G(-\frac 12\sul (-\Delta_\sigma)^{-1/2} + \frac{1+s}2)}{\G(-\frac 12 \sul (-\Delta_\sigma)^{-1/2} + \frac{1-s}2)},
\end{equation}
where $\G(z) = \int_0^\infty t^{z-1} e^{-t} dt$ denotes Euler gamma function, and $\sul$ is as in \eqref{sullaH}. The pseudodifferential operator $\mathscr L_s$ totally differs from $(-\sul)^s$, for which no geometry is involved. The two operators only coincide in the limit as $s\nearrow 1$. We have in fact, on one hand $(-\sul)^s\longrightarrow - \sul$ as $s\nearrow 1$. On the other hand, \eqref{bfmH}  gives as $s\nearrow 1$
\[
\mathscr L_s\ \longrightarrow\ 2 (-\Delta_\sigma)^{1/2} \frac{\G(-\frac 12\sul (-\Delta_\sigma)^{-1/2} + 1)}{\G(-\frac 12 \sul (-\Delta_\sigma)^{-1/2})} = 2 (-\Delta_\sigma)^{1/2}(-\frac 12)\sul (-\Delta_\sigma)^{-1/2} = - \sul.
\]
In their work \cite{FGMT} Frank, del Mar Gonz\'alez,  Monticelli and Tan introduced and solved a new remarkable extension problem of Caffarelli-Silvestre type (but very different from it!) for the operator $\mathscr L_{s}$, see also the works of Roncal and Thangavelu \cite{RT}, \cite{RT2} for a parabolic version of it. In the paper \cite{GTpot}, see also the companion works \cite{GT}, \cite{GTjam}, the first named author and Tralli constructed the following heat kernel for the Frank, del Mar Gonz\'alez, Monticelli and Tan extension problem in any group of Heisenberg type
\begin{align}\label{hk}
q_{(s)}((z,\sigma),t,y) & =  \frac{2^k}{(4\pi t)^{\frac{m}2 +k +1-s}} \int_{\R^k} e^{- \frac it \langle \sigma,\la\rangle}   \left(\frac{|\la|}{\sinh |\la|}\right)^{\frac m2+1-s}
\\
\notag
& \times e^{-\frac{|z|^2 +y^2}{4t}\frac{|\la|}{\tanh |\la|}} d\la.
\notag
\end{align}
The reader should note the critical appearance in \eqref{hk} of the fractal dimension $Q + 2(1-s)$ in the factor $t^{\frac{m}2 +k +1-s} = t^{\frac{Q+2(1-s)}2}$. Now, from general parabolic theory, we know that the heat kernel for the operator $\mathscr L_s$ is given by function $\mathscr K_{(s)}(z,\sigma,t) = (4\pi t)^{1-s} q_{(s)}((z,\sigma),t,0)$. Again parabolic theory says that $\mathscr E_{(s)}(z,\sigma) = \int_0^\infty \mathscr K_{(s)}(z,\sigma,t) dt$ is the fundamental solution of the pseudodifferential operator \eqref{bfmH}. Despite the fact that there exists no hint of the magic function \eqref{gauge} in the heat kernel \eqref{hk},  the following result holds, see \cite[Theorem 1.2]{GT}.

\medskip

\noindent \textbf{Theorem C.}\
\emph{Let $\bG$ be a group of Heisenberg type.  For any $0<s\le 1$ the fundamental solution of \eqref{bfmH} with pole at the group identity $e\in \bG$ and vanishing at infinity is given by
\begin{equation}\label{Esottos}
\mathscr E_{(s)}(z,\sigma) = \frac{C_{s}}{N(x)^{Q-2s}},
\end{equation}
where
\begin{equation}\label{C}
C_{s} = \frac{2^{\frac m2 + 2k-3s-1} \G(\frac 12(\frac m2+1-s)) \G(\frac 12(\frac m2 + k -s))}{\pi^{\frac{m+k+1}2} \G(s)}.
\end{equation}}

\medskip

Theorems A, B and C provide compelling evidence of the relevance of the gauge function $N(x)$ in \eqref{gauge} in the analysis of groups of Heisenberg type. In this note we are interested in characterising the level sets of such function via two overdetermined boundary value problems involving \eqref{pH}. The former can be stated as follows. Given a connected, bounded open set $\Om\subset \bG$, and a number $1<p<\infty$, we consider a solution to
\begin{equation}\label{i3}
\begin{cases}
\Delta_{H,p} f = - |z|^p\ \ \ \ \ \ \ \ \text{in}\ \Om,
\\
f = 0\ \ \ \ \ \ \ \ \ \ \ \ \ \ \ \ \ \ \ \ \text{on}\ \p \Om.
\end{cases}
\end{equation}
The non-homogeneous Dirichlet problem \eqref{i3} admits a unique nonnegative weak solution $f\in {\overset{\circ}{W}}_{H}^{1,p}(\Om)$, see \cite{D}. By this we mean that for every $\vf\in {\overset{\circ}{W}}_{H}^{1,p}(\Om)$ we have
\begin{equation}\label{ws}
\int_{\Om} |\nh f|^{p-2} \sa \nh f,\nh \vf\da dx = \int_{\Om} |z|^p \vf dx.
\end{equation}
We pose the following.

\vskip 0.2in

\noindent \textbf{Problem 1:} \emph{Let $f\ge 0$ be the weak solution to \eqref{i3}. Suppose in addition that  there exist $c\ge 0$ such that (in the measure theoretic sense of \eqref{gradconv}
 below)
\begin{equation}\label{od}
|\nh f|_{\big|\p\Om} = c |z|.
\end{equation}
Is it true that (up to a left-translation along the center $\exp \mathfrak v$ of the group $\bG$) $\Om$ is a ball centred at $e$ of the Koranyi gauge \eqref{gauge}?}

\vskip 0.2in

Before proceeding we explain in which sense we intend the boundary conditions to hold. Indicating with $h(x) = h(z,\sigma) = |z|$, we assume  that given $\ve>0$ there exists an open set $O = O(\ve) \supset \p \Om$ such that for some $c\ge 0$
\begin{equation}\label{gradconv}
||\nabla_H f(x)| - c\ h(x)|<\ve,\ \ \ \ f(x)<\ve,
\end{equation}
for a.e. $x\in \Om\cap O$ with respect to Lebesgue $(m+k)$-measure. This is a bi-invariant Haar measure in $\bG$.

\medskip

To provide the reader with an understanding of Problem 1 we start discussing its conjectured \emph{optimal geometric configuration}.
If $B_R = \{x\in \bG\mid N(x)<R\}$ is the gauge ball centred at the identity $e\in \bG$ and with radius $R$, then  problem \eqref{i3} admits an explicit solution.

\begin{proposition}\label{P:gauge}
The Dirichlet problem \eqref{i3} admits the following unique positive solution in $B_R$
\begin{equation}\label{solball}
f(x) = \frac{p-1}{2p(Q+p)^{\frac{1}{p-1}}} \left(R^{\frac{2p}{p-1}} - N(x)^{\frac{2p}{p-1}}\right),
\end{equation}
satisfying \eqref{od} with
\begin{equation}\label{cball}
c = \left(\frac{R^2}{Q+p}\right)^{\frac{1}{p-1}}.
\end{equation}
\end{proposition}
For the proof of this result see Section \ref{S:prelim} below. It is worth emphasising here that, when $p=2$, \eqref{solball} provides the following solution to \eqref{i3},
\begin{equation}\label{serrin}
f(x) = \frac{R^4 - N(x)^4}{4(Q+2)} = \frac{R^4 - |z|^4 - 16 |\sigma|^2}{4(Q+2)}.
\end{equation}
It should not be surprising that the function in \eqref{serrin} be real-analytic since, when $p=2$, such is the right-hand side in \eqref{i3}, and in every Metivier group, and therefore in particular in every group of Heisenberg type, $\sul$ is analytic hypoelliptic, see \cite{met1}, \cite{met2}. We also note that, when $p\not=2$, the intrinsic smoothness near the group identity $e\in \bG$ of the function in \eqref{solball} is the Folland-Stein H\"older class $\Gamma^{2,\alpha}$, with $\alpha = \frac{2}{p-1}$, which is better than the best possible local regularity $\Gamma^{1,\beta}$ known for solutions to $\Delta_{H,p} f = 0$, see \cite{zhong}, \cite{Gmanu2}, \cite{muzh} and the more recent contributions \cite{CCGme}, \cite{cittimu} and \cite{CCZ}. On the other hand, the function $f$ in \eqref{solball} is $C^\infty$ in a neighbourhood of $\p B_R$, and  moreover its Riemannian gradient $\nabla f$ does not vanish on the whole of $\p B_R$ (although by \eqref{nhball} below its horizontal gradient $\nh f$ does vanish on $\p B_R$!). This can be easily seen observing that if $\mathscr Z$ denotes the generator of the anisotropic group dilations in $\bG$, $\delta_\la(x) = (\la z,\la^2 \sigma)$, then we have from \eqref{solball} for every $x\not= e$,
\[
\mathscr Z f(x) = - \frac{p-1}{2p(Q+p)^{\frac{1}{p-1}}} \mathscr Z(N(x)^{\frac{2p}{p-1}}) = - \frac{1}{(Q+p)^{\frac{1}{p-1}}} N(x)^{\frac{2p}{p-1}}.
\]
This shows that $\mathscr Z f = - \frac{1}{(Q+p)^{\frac{1}{p-1}}} R^{\frac{2p}{p-1}}<0$ on $\p B_R$, and since $\mathscr Z f(x) = \langle \nabla f(x),\mathscr Z(x)\rangle$, we infer that, unlike $\nh f$, the Riemannian gradient $\nabla f$ does not vanish on $\p B_R$, not even on the \emph{characteristic set} $\Sigma = \{(0,\sigma)\mid |\sigma| = \frac{R^2}4\}$.
Since $f>0$ in $B_R$ and $f = 0$ on $\p B_R$, one therefore has $\nu = - \frac{\nabla f}{|\nabla f|}$.

\medskip

With this being said, we recall that in a group of Heisenberg type $\bG$ the function $N(x)$ in \eqref{gauge} satisfies the  two (important) identities \eqref{iiH} below.
If we let $f$ denote the function in \eqref{solball}, then a simple computation which uses \eqref{iiH} shows that for the gauge ball $B_R$ one has
\begin{equation}\label{nhball}
|\nh f(x)|^p = \frac{1}{(Q+p)^{\frac{p}{p-1}}} N(x)^{\frac{2p}{p-1}} |z|^p.
\end{equation}
If we now compare \eqref{nhball} with \eqref{solball}, we are lead to the discovery that, in the ball $B_R$, the function
\begin{equation*}\label{Pball}
P(x)  = |\nh f(x)|^p + \frac{2p}{(p-1)(Q+p)} f(x) |z|^p,
\end{equation*}
satisfies the following pointwise constraint
\begin{equation}\label{Pballcon}
P(x) = c^p |z|^p.
\end{equation}
It is clear that the identity \eqref{Pballcon} trivially implies the integral constraint
\begin{equation}\label{rigprop}
\int_{B_R} P(x) dx = c^p \int_{B_R} |z|^p dx,
\end{equation}
with $c>0$ as in \eqref{cball}.

These considerations lead us to introduce our first result, Theorem \ref{T:main} below. The latter shows that, if a \emph{general domain} $\Om\subset \bG$ admits a solution $f$ of \eqref{i3} satisfying the overdetermined condition \eqref{od}, then  the integral rigidity property \eqref{rigprop} continues to be valid for $\Om$.
\medskip

\begin{theorem}\label{T:main}
Let $f$ be the solution of the Dirichlet problem \eqref{i3} in a smooth domain $\Om\subset \bG$, and suppose that $f$ satisfy the overdetermined condition \eqref{od}. Suppose in addition that $f\in \G^{2,\alpha}(\overline \Om)$ for some $\alpha\in (0,1)$, and that $\nabla f$ (the Riemannian gradient of $f$) does not vanish on $\p \Om$. Then we must have $c>0$, and moreover the function defined in $\Om$ by the equation
\begin{equation}\label{Pomega}
P(x)  = |\nh f(x)|^p + \frac{2p}{(p-1)(Q+p)} f(x) |z|^p,
\end{equation}
 satisfies the following integral rigidity property
\begin{equation}\label{ii}
\int_{\Om} P(x) dx = c^p \int_{\Om} |z|^p dx.
\end{equation}
\end{theorem}

\medskip

The proof of Theorem \ref{T:main} will be given in Section \ref{S:main} below. Notice that since $f = 0$ on $\p \Om$ in view of \eqref{i3}, the overdetermined assumption \eqref{od} implies that $P(x) = c^p |z|^p$ on $\p \Om\setminus \Sigma$, where we have denoted by $\Sigma$ the characteristic set of $\Om$ (this is the set where the vector fields \eqref{Xi} become tangent to $\p \Om$). The conclusion \eqref{ii} guarantees that this is also true inside $\Om$, but in an integral sense.
In light of Theorem \ref{T:main} and of the above considerations, we next propose the following conjecture. In Proposition \ref{P:iips} below we show that such conjecture has a positive answer in several geometrically significant situations.

\vskip 0.2in

\noindent \textbf{Problem 2:} \emph{Let $\Om\subset \bG$ be a (smooth)  connected and bounded open set, and let $f$ be a solution to \eqref{i3}, satisfying \eqref{od}.
Is it true that the pointwise identity \eqref{Pballcon} hold at every point $x\in \Om$?}

\vskip 0.2in

Proposition \ref{P:gauge} raises the obvious problem whether, given a group of Heisenberg type $\bG$, there exist other bounded open sets in $\bG$ that support a solution to the overdetermined problem \eqref{i3}, \eqref{od}. To address this issue,
we indicate with $O(\mathfrak h)$ the orthogonal group of the vector space $\mathfrak h$, where as above $\bg = \mathfrak h \oplus \mathfrak v$ is the Lie algebra of $\bG$. Now, unlike the standard Laplacian, the partial differential operators \eqref{sullaH} or \eqref{pH} are not invariant with respect to action of $O(\mathfrak h)$. This leads us to identify a subgroup of $O(\mathfrak h)$ which is directly connected to a basic invariance property, expressed by Proposition \ref{L:Hlapinv} below, of the partial differential equations of interest in this work. Henceforth, we denote by $\mathscr S(\mathfrak h)$ the subgroup of those transformations $S\in O(\mathfrak h)$ such that for every $\sigma\in \mathfrak v$
\begin{equation}\label{commell}
SJ(\sigma) = J(\sigma) S.
\end{equation}
We will refer to  $\mathscr S(\mathfrak h)$ as the $J-$\emph{invariant} orthogonal group of $\Vone$.  In the next result we indicate with $\Delta_{H,\infty}$ the fully nonlinear differential operator defined by \eqref{LapHinfty} below.

\begin{proposition}\label{L:Hlapinv}
Let $S\in \mathscr S(\mathfrak h)$. Given  a function $f:\bG\to \R$, set
$g(z,\sigma) = f(Sz,\sigma)$.
One has
\begin{equation}\label{invlap}
\sul g(z,\sigma) = \sul f(Sz,\sigma),
\end{equation}
and also
\begin{equation}\label{invlapinfty}
\Delta_{H,\infty} g(z,\sigma) = \Delta_{H,\infty} f(Sz,\sigma).
\end{equation}
Finally, for any $1<p<\infty$ we have
\begin{equation}\label{invplap}
\Delta_{H,p} g(z,\sigma) = \Delta_{H,p} f(Sz,\sigma).
\end{equation}
\end{proposition}

Proposition \ref{L:Hlapinv} plays a pervasive role in our next result, Theorem \ref{T:glps} below, which provides a partial solution to Problem 1. Before stating it, we need to introduce the following.

\begin{definition}\label{D:trans}
We say that a group of Heisenberg type $\bG$ has the \emph{Property (H)} if the subgroup $\mathscr S(\mathfrak h)\subset O(\mathfrak h)$ acts transitively on the unit sphere in $\mathfrak h$.
\end{definition}

It is clear that Property (H) is of a fundamental nature and, in fact, it has been deeply investigated by several people, see \cite{MS}, \cite{AS}, \cite{Ber} and \cite{Sim}. The most basic non-Abelian instance is when $\bG = \Hn$, the Heisenberg group. Identifying $\mathfrak h \cong \R^{2n} \cong \mathbb C^n$, and indicating with $Sp_n(\R)$ the symplectic group in $\R^{2n}$, one has
\begin{equation}\label{Hn}
\mathscr S(\mathfrak h) = O(2n)\cap Sp_n(\R) = U(n),
\end{equation}
the unitary group. It is well known that $U(n)$ acts transitively on the unit sphere in $\mathbb C^n$.
While the reader is referred to Section \ref{S:cyl} for a detailed discussion, here we confine ourselves to mention that not every group of Heisenberg type $\bG$ has the Property (H). For instance, the 15-dimensional octonionic Heisenberg group does not have this property. In fact, as the next result shows, this is the only exception among all Iwasawa groups. We recall that such groups can be described as the boundaries of the  real, complex, quaternionic or octonionic hyperbolic space, see \cite[Section 9.3]{Mo} and also \cite[Section 19]{Pansu89}.

\begin{proposition}\label{P:iwa}
Let $\bG$ be a group of Iwasawa type. Then, excluding the unique octonionic Heisenberg group of dimension 15, $\bG$ has the \emph{Property (H)}. This property fails for the octonionic Heisenberg group.
\end{proposition}

We will also show that there exist groups of Heisenberg type, which are not Iwasawa, that have the Property (H).
From the perspective of the present work the relevance of such property is connected to a basic symmetry result for solutions of the problem \eqref{i3}, \eqref{od}. We need to introduce the relevant definition.

\begin{definition}\label{D:psym}
Given a group of Heisenberg type $\bG$, we say that an open set $\Om\subset \bG$ has \emph{partial symmetry} if in the space $\R\times \R^{k}$, with coordinates $(\xi,\sigma)$, there exists an open set $\Om^\star$ such that $
(\xi,\sigma)\in \Om^\star\ \Longleftrightarrow\ (-\xi,\sigma)\in \Om^\star,
$
and
\begin{equation}\label{Ommani}
\Om = \{(z,\sigma)\in \bG\mid  (\frac{|z|^2}{4},\sigma)\in \Om_+^\star\}.
\end{equation}
\end{definition}
In \eqref{Ommani} we have denoted $\Om^\star_+ = \{(\xi,\sigma)\in \Om^\star\mid \xi\ge 0\}$. It is important to notice that the transformation
\[
(z,\sigma)\in \bG\ \longrightarrow\ (\frac{|z|^2}{4},\sigma)\in \R\times \R^{k},
\]
maps the gauge ball $B_R = \{(z,\sigma)\in \bG\mid |z|^4 + 16 |\sigma|^2 < R^4\}$ onto half of the Euclidean ball $B^\star_{+}(r) = \{(\xi,\sigma)\in \R\times \R^{k}\mid \xi^2 + |\sigma|^2 < r^2,\ \xi\ge 0\}$, with $r = \frac{R^2}{4}$.
We also stress that in Definition \ref{D:psym} no symmetry is requested in the vertical variable $\sigma\in \R^k$.

The next result shows that, when the group has the Property (H), the gauge balls are the unique bounded open sets with partial symmetry in which there exist a solution to \eqref{i3}, satisfying the overdetermined condition \eqref{od}. This provides a partial solution to Problem 1.

\medskip

\begin{theorem}\label{T:glps}
Suppose that $\bG$ have the \emph{Property (H)}, and let $\Om\subset \bG$ be a connected bounded open set having partial symmetry. Let $1<p<\infty$ and assume that $f$ solve \eqref{i3} and satisfy the overdetermined condition \eqref{od}. Then there exist $R>0$ and $\sigma_0\in \R^k$ such that $\Om$ is a ball of the Koranyi-Kaplan gauge \eqref{gauge} centred at $(0,\sigma_0)$ with radius $R$, i.e.
\[
\Om = B_R(0,\sigma_0) = \{(z,\sigma)\in \bG\mid |z|^4 + 16 |\sigma-\sigma_0|^2 < R^4\},
\]
and $f$ is given by left-translating \eqref{solball} by $(0,-\sigma_0)$.
\end{theorem}

\medskip

The proof of Theorem \ref{T:glps} will be given in Section \ref{S:cyl} below.  At the core of our argument there is Proposition \ref{P:pcylG} that, via the geometric invariances in Lemmas \ref{L:Hlapinv} and \ref{L:gradinv}, allows to connect the sub-Riemannian problem \eqref{i3}, \eqref{od}, to a (Euclidean) 1989 theorem for the classical
$p$-Laplacian of J. Lewis and the first named author, see Theorem D in Section \ref{S:ops} below.

The next result, whose proof will be given in Section \ref{S:cyl} below, shows that  besides the case of the gauge ball discussed above, when the group has the Property (H) the answer to Problem 2 is positive when the domain $\Om$ has partial symmetry.

\medskip

\begin{proposition}\label{P:iips}
Assume that $\bG$ have the \emph{Property (H)}. Let $\Om \subset \bG$ be a smooth domain with partial symmetry, and $f$ be a solution to \eqref{i3}, \eqref{od}. Then the function $P(x)$ defined in \eqref{Pomega} satisfies the pointwise identity $P(x) = c^p |z|^p$, $x\in \Om$.
\end{proposition}

A second overdetermined problem which we propose is the following. We recall that given an open set $\Om\subset \bG$, and $\varnothing \not= K\subset \Om$ a compact set, the couple $(K,\Om)$ is called a condenser. One says that the condenser is \emph{ringlike} if both $K$ and $\bG\setminus \Om$ are  connected. In such case $\Om\setminus K$ is called a \emph{ring}. For a given condenser $(K,\Om)$, we let
\[
\mathscr F(K,\Om) = \left\{\vf\in C^\infty_0(\Om)\mid \vf \geq 1 \ \text{on}\ K\right\}.
\]
\begin{definition}\label{D:capacity}
For $1\leq p < \infty$ the $p$-\emph{capacity} of the condenser $(K,\Om)$ is defined as follows
\[
\operatorname{cap}_p (K,\Om) = \inf_{\vf\in \mathscr F(K,\Om)} \int_\Om |\nh \vf|^p dx.
\]
When $\Om = \bG$, then we simply write $cap_{H,p} K$, instead of $\operatorname{cap}_{H,p}(K,\bG)$.
\end{definition}

If we let
\[
\mathscr P(K,\Om) = \left\{\vf\in C^\infty_0(\Om)\mid \vf \equiv 1 \ \text{in a neighborhood of}\ K, 0 \leq \vf \leq 1\right\},
\]
then, similarly to the classical case, see \cite{M} p. 100, one recognises that
\begin{equation}\label{eqdef}
\operatorname{cap}_{H,p}(K,\Om) = \inf_{\vf\in \mathscr P(K,\Om)} \int_\Om |\nh \vf|^p dx.
\end{equation}
One can replace the compact set $K$ with any bounded open set $\omega\subset \bG$, see \cite{D}. When $\Om = \bG$, we write $\operatorname{cap}_{H,p} \omega = \operatorname{cap}_{H,p}(\omega,\bG)$, and  call this number the horizontal $p$-capacity of $\omega$.

Given $1<p<Q$, the $p$-capacitary problem consists in finding a $p$-harmonic function  $f$ in $\bG\setminus \overline \Om$, such that $f = 1$ on $\p \Om$, and $f=0$ at infinity, i.e.,
\begin{equation}\label{icap}
\begin{cases}
\Delta_{H,p} f = 0\ \ \ \ \ \ \ \ \ \ \text{in}\ \bG\setminus \overline{\Om},
\\
f = 1\ \ \ \ \ \ \ \ \ \ \ \ \ \ \ \ \text{on}\ \p \Om,
\\
f = 0\ \ \ \ \ \ \ \ \ \ \ \ \ \ \ \ \ \text{at}\ \infty.
\end{cases}
\end{equation}
As it is customary, the boundary condition $f = 1$ on $\p \Om$ is interpreted by assuming that $f-1$ is zero on $\p \Om$ in the weak sense. Using variational tools, in the preprint \cite{DG} the existence of a unique (weak) solution $f$ to the problem \eqref{icap} was established in any Carnot group. Recently, this construction has been generalised to metric measure spaces by Bonk, Capogna and Zhou in \cite{BCZ}.

Given $f$ as in \eqref{icap}, the horizontal $p$-capacity of $\Om$ is then given by
\begin{equation}\label{pcap}
\operatorname{cap}_{H,p}(\Om) = \int_{\bG\setminus \overline{\Om}} |\nh f|^p dx.
\end{equation}
For such reason, $f$ is called the $p$-\emph{capacitary potential} of $\Om$.
We pose the following.

\medskip

\noindent \textbf{Problem 3:} \emph{Let $1<p<Q$ and $\Om\subset \bG$ be a connected open set. Let $f$ be a solution to \eqref{icap} and suppose in addition that  there exist $c> 0$ such that, in the sense specified in \eqref{gradconv},
\begin{equation}\label{odd}
|\nh f|_{\big|\p\Om} = c |z|.
\end{equation}
Is it true that (up to a left-translation along the center) $\Om$ is a ball centred at $e$ of the Koranyi-Kaplan gauge \eqref{gauge}?}

\medskip

We again analyse the conjectured optimal geometric configuration for Problem 3. When $\Om = B_R$, a gauge ball centred at $e\in \bG$, there exists an explicit solution to the $p$-capacitary problem. Using \eqref{fs} it is clear that when $\Om = B_R$ one has the following explicit solution of the capacitary problem \eqref{icap}
\begin{equation}\label{sol}
f(x) = \left(\frac{R}{N(x)}\right)^{\frac{Q-p}{p-1}}.
\end{equation}

\medskip

Similarly to Theorem \ref{T:glps} above, the next results shows that Problem 3 does have a positive answer when the inner domain has partial symmetry.

\begin{theorem}\label{T:gsps}
Suppose that $\bG$ have the \emph{Property (H)}. Assume that $\Om\subset \bG$ be a connected bounded open set having partial symmetry. Suppose $1<p<Q$ and that $f$ be a weak solution to \eqref{icap} satisfying the overdetermined condition \eqref{gradconv}. Then there exist $R>0$ and $\sigma_0\in \R^k$ such that $\Om$ is a ball of the Koranyi-Kaplan gauge \eqref{gauge} centred at $(0,\sigma_0)$ with radius $R$,
\[
\Om = B_R(0,\sigma_0) = \{(z,\sigma)\in \bG\mid |z|^4 + 16 |\sigma-\sigma_0|^2 < R^4\},
\]
and $f$ is given by left-translating \eqref{sol} by $(0,-\sigma_0)$.
\end{theorem}

\medskip

Theorem \ref{T:gsps} will be proved in Section \ref{S:cyl} below. Similarly to Theorem \ref{T:glps}, it will be deduced from Proposition \ref{P:pcylG}, Lemmas \ref{L:gradinv} and \ref{L:Hlapinv}, and from the Euclidean Theorem E, which we have recalled in the appendix in Section \ref{S:a} below. The reader might be left with wondering about a potential discrepancy between the hypothesis $1<p<Q$ in the statement of Theorem \ref{T:gsps} versus the assumption $1<p<n$ in Theorem E. As the reader will see, there is no such discrepancy: the reason for this is that what will play the role of the dimension $n$ in Theorem E is the possibly fractal dimension $n = \frac{m+p}2 + k = \frac{Q+p}2$. Therefore, the restriction $p<  \frac{Q+p}2 = n$, that is needed to implement Theorem E, is in fact equivalent to the condition $p<Q$ in Theorem \ref{T:gsps}!

Some final comments are in order. With $f$ given by \eqref{sol}, consider  the function in $\bG\setminus\overline B_R$ defined by the equation
\begin{equation}\label{Pext}
P(x) = P(f;x) = \frac{|\nh f(x)|^p}{f(x)^{\frac{p(Q+p-2)}{Q-p}}}
\end{equation}
(the function $P$ is well defined since $f>0$). A direct computation allows to recognise that such function satisfies in $\bG\setminus\overline B_R$ the pointwise identity
\begin{equation}\label{Pcapball}
P(x) = c^p |z|^p,
\end{equation}
with $c = \frac{Q-p}{p-1} \frac{1}{R^2}$. Let us now assume that $\Om\subset \bG$ be a connected bounded open set, and suppose that $f$ be its capacitary potential, i.e., that $f$ solves \eqref{icap}. For such $f$, consider $P$ defined by \eqref{Pext} (this function is well defined since $f>0$ by the Harnack inequality in \cite{CDG}). One important property of $P$ is the following scale invariance: let $f_\la(x) = \la^{\frac{Q-p}{p-1}} f(\delta_\la x)$, where $\delta_\la$ denotes the non-isotropic group dilations. Then the reader can verify that
\[
P(f_\la;x) = P(f;\delta_\la x).
\]

\medskip

We pose the following.

\medskip

\noindent \textbf{Problem 4:} \emph{Does the validity of the equation \eqref{Pcapball} in $\bG\setminus \Om$ characterise a gauge ball?}

\medskip

Concerning Problem 4 we mention that, using the work in the cited paper \cite{GS}, we have a positive answer in domains with partial symmetry. We mention that, analogously to what was proved in the Euclidean setting in \cite{GS}, the validity of \eqref{Pcapball} is connected to the following information on the horizontal mean curvature of $\p \Om$  outside the characteristic portion of $\p \Om$
 \begin{equation}\label{Hgaugeball}
\mathscr H =  \frac{|z|}{(Q+p) c^{p-1}}.
\end{equation}
Whether in a group of Heisenberg type $\bG$ the equation \eqref{Hgaugeball}  does characterise a gauge ball centred at $e$ presently remains a challenging open question. In Proposition \ref{P:bubbleps} we show that, if $\Om$ is a $C^2$, connected, bounded domain with partial symmetry in a group of Heisenberg type $\bG$, the assumption
\begin{equation}\label{alphaH}
\mathscr H = \alpha |z|,
\end{equation}
 with $\alpha>0$, does imply that $\Om$ must be a gauge ball $B_R(0,\sigma_0)$. With a completely different approach from us, in their interesting recent paper \cite{GMT}, Guidi, Martino and Tralli have proved that for a smooth domain $\Om\subset \mathbb H^1$ with two isolated characteristic points on the vertical axis, the hypothesis \eqref{alphaH} does characterise the Koranyi ball. As a corollary of their umbilical Theorem 1.2 they have also shown that, for the higher  Heisenberg groups $\Hn$, this continues to be true if one makes the a priori assumption of partial symmetry on the ground domain.

\medskip

\textbf{Acknowledgement.} We thank Professor Robert Bryant for kindly providing the computation of the centralizer of  Spin(7) $\subset$ SO(8) used in the proof of Proposition \ref{P:iwa}.

\vskip 0.2in


\section{Some preparatory results}\label{S:prelim}

In this section we collect some background material that will be needed in the rest of the paper. We also prove Proposition \ref{P:gauge}, together with some useful new results. A simply connected real Lie group $\bG$ is called a Carnot group of step two if its  Lie algebra is stratified and two-nilpotent, i.e., $\bg =
\mathfrak h\oplus \mathfrak v$, where $[\mathfrak h,\mathfrak h] = \mathfrak v, [\mathfrak h,\mathfrak v] = \{0\}$ (we refer to the seminal paper \cite{Fo2} for an introduction to Carnot groups in general, see also \cite[Chapter 2]{Gparis}). We equip $\bg$ with an inner product with respect to which $\{e_1,...,e_m\}$ and $\{\ve_1,...,\ve_k\}$ denote an orthonormal basis of $\mathfrak h$ and  $\mathfrak v$, respectively.
Consider the analytic mappings $z:\bG\to \mathfrak h$, $\sigma:\bG\to \mathfrak v$ uniquely defined through the equation $x = \exp(z(x)+\sigma(x))$. For each $i=1,...,m$ we set
\[
z_i = z_i(x) = \sa z(x),e_i\da,
\]
whereas for $s=1,...,k$ we let
\[
\sigma_s = \sigma_s(x) = \sa\sigma(x),\ve_s\da.
\]
We will indicate with $(z,\sigma)\in \bg$ the logarithmic coordinates of a point $x\in \bG$.
Given $z, \sigma\in \bg$, the Baker-Campbell-Hausdorff formula presently reads
\begin{equation}\label{BCH}
\exp(z) \circ \exp(\sigma) = \exp{\big(z + \sigma + \frac{1}{2}
[z,\sigma]\big)},
\end{equation}
see \cite[Sec. 2.15]{V}. Notice that \eqref{BCH} assignes a group law $\circ$ in $\bG$, which is noncommutative. The multiplication $x \circ y$ in $\bG$ is obtained from \eqref{BCH} by the algebraic commutation relations between the elements of its Lie algebra, see \eqref{grouplaw} above.

The group $\bG$ is naturally equipped with a one-parameter family of automorphisms $\{\delta_\lambda\}_{\lambda>0}$  which are called the group dilations. One first defines a family of non-isotropic dilations $\Delta_\lambda :\bg \to \bg$ in the Lie algebra by assigning the formal degree $j$ to the $j$-th layer $V_j$ in the stratification of $\bg$. This means that if $z\in \mathfrak h$, $\sigma\in \mathfrak v$ one lets
\begin{equation}\label{dilg}
\Delta_\lambda(z+\sigma) = \lambda z + \lambda^2 \sigma.
\end{equation}
One then uses the exponential mapping to define a one-parameter family of group automorphisms $\delta_\la :\bG \to \bG$ by the equation
\begin{equation}\label{dilG}
\delta_\lambda(x) = \exp \circ \Delta_\lambda \circ \exp^{-1}(x),\quad\quad\quad x\in \bG.
\end{equation}
We assume henceforth that $\bG$ is endowed with a left-invariant Riemannian
tensor with respect to which the vector fields $X_1,\ldots,X_m$,
$T_1,\ldots,T_k$ defined in \eqref{vfields} are orthonormal at every point. If $f:\bG\to \R$ is a smooth function, its horizontal gradient is given by
\begin{equation}\label{sh}
\nh f = \sum_{i=1}^m X_i f\ X_i.
\end{equation}
This is the projection of the Riemannian gradient of $f$,
\begin{equation}\label{rg}
\nabla f = \sum_{i=1}^m X_i f\ X_i + \sum_{\ell = 1}^k T_\ell f\ T_\ell,
\end{equation}
onto the horizontal bundle $H = \exp \mathfrak h$.
In the logarithmic coordinates $(z,\sigma)$ one has
\begin{equation}\label{Xi}
X_i  = \p_{z_i} + \frac 12 \sum_{\ell=1}^k \sa J(\ve_\ell)z,e_i\da \partial_{\sigma_\ell}.
\end{equation}
We note that when $f:\bG\to \R$ is a smooth function, one has the following elementary consequence of \eqref{Xi}
\begin{equation}\label{grad}
\sa \nh f(x),z\da = \sa \nabla_z f(x),z\da.
\end{equation}
We have in fact
\begin{align*}
\sa \nh f(x),z\da & = \sum_{i=1}^m X_i f(x) z_i = \sum_{i=1}^m\p_{z_i}f(x) z_i + \frac 12 \sum_{\ell=1}^k \sum_{i=1}^m \sa J(\ve_\ell)z,e_i\da \sa z,e_i\da \partial_{\sigma_\ell} f(x)
\\
& = \sa \nabla_z f(x),z\da + \frac 12 \sum_{\ell=1}^k \sa J(\ve_\ell)z,z\da \partial_{\sigma_\ell} f(x) = \sa \nabla_z f(x),z\da.
\end{align*}

The horizontal Laplacian with respect to the basis $\{e_1,...,e_m\}$ is given by \eqref{sulla}, and when $\bG$ is of Heisenberg type such differential operator takes the special expression \eqref{sullaH}.
The following result will be useful, see \cite[formulas (2.7) \& (2.10)]{CDGcap}.

\begin{lemma}\label{L:Xec}
Let $\bG$ be a group of Heisenberg type. Then
\begin{equation}\label{iiH}
|\nh N| = \frac{|z|}{N},\ \ \ \ \ \ \ \ \ \ \ \ \ \ \ \sul N = \frac{Q-1}{N} |\nh N|^2.
\end{equation}

\end{lemma}

Consider now for  functions $u:\bG\to \R$ and $F:\R\to \R$, the composition  $f=F\circ u$. An elementary computation based on the chain rule gives
\begin{equation}\label{spilcomp}
\Delta_{H,p} f = F'(u) |F'(u)|^{p-2} \Delta_{H,p} u + (p-1) |F'(u)|^{p-2} F''(u) |\nh u|^p.
\end{equation}

If we apply \eqref{spilcomp} with the choice $u(x) = N(x)$, and use \eqref{iiH}, we obtain the following result, see \cite[Lemma 2.4]{CDGcap}.

\begin{lemma}\label{L:Lpradial}
For every $1<p<\infty$ one has
\begin{equation}\label{Lpradial}
\Delta_{H,p} F(N) = (p-1)  \frac{|F'(N)|^{p-2}}{N^p} \left[F''(N) + \frac{Q-1}{p-1} \frac{F'(N)}{N}\right] |z|^p.
\end{equation}
\end{lemma}
Note that \eqref{Lpradial} implies in particular that $\Delta_{H,p} F(N)$ is not a function of $N$. With Lemma \ref{L:Lpradial} in hands, it is not difficult to now give the
\begin{proof}[Proof of Proposition \ref{P:gauge}]
From the equation \eqref{Lpradial} it is clear that if we want to solve the PDE
\[
\Delta_{H,p} f = - |z|^p
\]
in the gauge ball $B_R$ with boundary condition $f = 0$ on $\p B_R$, then with $f = F \circ N$ we must solve the ODE
\[
(p-1)  \frac{|F'(N)|^{p-2}}{N^p} \left[F''(N) + \frac{Q-1}{p-1} \frac{F'(N)}{N}\right] = - 1,
\]
on the interval $[0,R]$, with condition $F(R) = 0$. We try for a solution in the form $F(N) = \la N^{\alpha}$, with $\la\in \R$ and $\alpha>0$ to be chosen.
Since for such a choice an elementary computation gives
\begin{align*}
& (p-1)  \frac{|F'(N)|^{p-2}}{N^p} \left[F''(N) + \frac{Q-1}{p-1} \frac{F'(N)}{N}\right] = \la^{p-1} \alpha^{p-1} \left[(\alpha-1)(p-1) + (Q-1)\right] N^{\alpha p - 2p - \alpha},
\end{align*}
it is clear that we must choose $\alpha p - 2p - \alpha = 0$, so that
\[
\alpha = \frac{2p}{p-1}.
\]
Once this is done, we want
\[
\la^{p-1} \alpha^{p-1} \left[(\alpha-1)(p-1) + (Q-1)\right] = - 1,
\]
and this forces us to choose
\[
\la = - \frac{p-1}{2p(Q+p)^{\frac{1}{p-1}}}.
\]
The condition $F(R) = 0$ finally gives $f = F(N)$ in the form \eqref{solball}. From $f = \la N^\alpha$, the chain rule, and from the former identity in \eqref{iiH}, it is now clear that $f$ satisfies the overdetermined condition \eqref{od} on $\p B_R$ with
\[
c = |\la| \alpha R^{\alpha-2} = \left(\frac{R^2}{Q+p}\right)^{\frac{1}{p-1}}.
\]
This completes the proof.

\end{proof}

Our next step is to analyse the action of the operator $\Delta_{H,p}$ on another notable class of functions. We begin with the relevant definition.

\begin{definition}\label{D:fparsym}
We say that $f: \bG\to \R$ has \emph{partial symmetry} if, denoting by $(\xi,\sigma)$ the variable point in the space $\R\times \R^k$, there exists $g:\R\times \R^k \to \R$, even in $\xi$, such that
\begin{equation}\label{fgG}
f(z,\sigma) = g(\frac{|z|^2}{4},\sigma).
\end{equation}
\end{definition}
We stress that in Definition \ref{D:fparsym} no symmetry is a priori assumed in the central variable $\sigma$.
It follows from \eqref{kap} that when $f$ has partial symmetry, then for every $\ell = 1,...,k$ one has
\begin{equation}\label{thetal}
\Theta_\ell f = 0.
\end{equation}
In fact,  \eqref{fgG} give
\[
\Theta_\ell f = \sum_{i=1}^m \sa J(\ve_\ell)z,e_i\da \p_{z_i} f = \frac{g_\xi}2 \sum_{i=1}^m \sa J(\ve_\ell)z,e_i\da \sa z,e_i\da = \frac{g_\xi}2 \sa J(\ve_\ell)z,z\da = 0.
\]
We also mention that the proof of the next Lemma \ref{L:nice} could also be directly extracted from \cite[Propositions 3.2 \& 3.3]{GVduke}. However, we insist in providing a direct proof since the special form of the right-hand side of \eqref{fgG} is tailor-made on the geometry of the problems at study in the present work, and plays a critical role in the computations.

\begin{lemma}\label{L:nice}
Let $f$ be a function as in \eqref{fgG} in a group of Heisenberg type $\bG$. If $\xi = \frac{|z|^2}{4}\ge 0$, then
\begin{equation}\label{gradG}
|\nh f|^2 = \xi\ (g_{\xi}^2 + |\nabla_\sigma g|^2),
\end{equation}
and
\begin{equation}\label{bg2}
\Delta_H f = \xi\ \left(g_{\xi\xi} + \frac{m}{2\xi}\ g_\xi+ \Delta_\sigma  g\right).
\end{equation}
\end{lemma}

\begin{proof}
To prove \eqref{gradG} we observe that \eqref{Xi} gives
\begin{align*}
|\nh f|^2 & = \sum_{i=1}^m \left(g_\xi \frac{z_i}2  + \frac 12 \sum_{\ell=1}^k \sa J(\ve_\ell)z,e_i\da g_{\sigma_\ell}\right)\left(g_\xi \frac{z_i}2  + \frac 12 \sum_{\ell'=1}^k \sa J(\ve_\ell')z,e_i\da g_{\sigma_{\ell'}}\right)
\\
& = g_\xi^2 \frac{|z|^2}4 + \frac{g_\xi}2 \sum_{\ell=1}^k \sum_{i=1}^m \sa J(\ve_\ell)z,e_i\da \sa z,e_i\da g_{\sigma_{\ell}} + \frac 14 \sum_{\ell, \ell'=1}^k \sum_{i=1}^m \sa J(\ve_\ell)z,e_i\da \sa J(\ve_\ell')z,e_i\da g_{\sigma_\ell}g_{\sigma_{\ell'}}
\\
& = g_\xi^2 \frac{|z|^2}4 + \frac{g_\xi}2 \sum_{\ell=1}^k \sa J(\ve_\ell)z,z\da g_{\sigma_{\ell}} + \frac 14 \sum_{\ell, \ell'=1}^k \sa J(\ve_\ell)z,J(\ve_\ell')z\da g_{\sigma_\ell}g_{\sigma_{\ell'}}
\\
& = g_\xi^2 \frac{|z|^2}4  + \frac{|z|^2}4 \sum_{\ell, \ell'=1}^k  \delta_{\ell \ell'}g_{\sigma_\ell}g_{\sigma_{\ell'}}
\end{align*}
where in the last equality we have used \eqref{ip}. Since $\xi = \frac{|z|^2}4$, it is now clear that \eqref{gradG} follows.

To establish \eqref{bg2} note that from \eqref{sullaH} and \eqref{thetal} we infer that, if $\bG$ is of Heisenberg type, then for a function for which \eqref{fgG} hold, one has
\begin{equation}\label{bg}
\Delta_H f = \Delta_z f + \frac{|z|^2}{4} \Delta_\sigma f.
\end{equation}
The equation \eqref{bg} takes an interesting form if we express it in terms of the function $g$. In fact, an elementary computation shows that if $f$ is as in \eqref{fgG}, then with $\xi = \frac{|z|^2}4$ we have
\begin{equation}\label{laplap}
\Delta_z f = \xi \left(g_{\xi\xi} + \frac{m}{2\xi}\ g_\xi\right).
\end{equation}
Combining \eqref{bg} with \eqref{laplap}, we conclude that \eqref{bg2} does hold.

\end{proof}

\begin{remark}\label{R:lap}
In connection with \eqref{bg2} we emphasise that, if we think of $(w,\sigma)\in \R^{\frac m2+1}\times \R^k$, then the standard Laplacian in such space, acting on functions which depend on $\xi = |w|$ and $\sigma\in \R^k$, is precisely the operator
\begin{equation}\label{cylDel}
\Delta_{(w,\sigma)} = \p_{\xi\xi} + \frac{m}{2\xi}\ \p_\xi + \Delta_\sigma.
\end{equation}
It is worth noting in connection with \eqref{cylDel} that in a group of Heisenberg type the complex structure induced by the Kaplan's mapping $J$ in \eqref{kap} forces $m = 2m_1$ for some $m_1\in \mathbb N$. Thus, the number $\frac m2+1$ is always an integer.
\end{remark}

Next, we establish an interesting generalisation of \eqref{bg2}.

\begin{proposition}\label{P:pcylG}
Let $\bG$ be a group of Heisenberg type and $1<p<\infty$. If $\xi = \frac{|z|^2}4$, then on a function $f$ as in \eqref{fgG} the nonlinear operator \eqref{pH} is given by the formula
\begin{equation}\label{pisub2}
\Delta_{H,p} f  = \xi^{\frac p2} |\nabla g|^{p-2} \left[g_{\xi\xi} + \frac{m+ p -2}{2\xi} g_\xi + \Delta_\sigma g + (p-2) \frac{\Delta_\infty g}{|\nabla g|^2}\right].
\end{equation}
\end{proposition}

\begin{proof}
We begin by observing that \eqref{pH} can be alternatively (formally) expressed by
\begin{equation}\label{plaplap}
\Delta_{H,p} f = |\nabla_H f|^{p-2} \left[\Delta_H f + (p-2) \frac{\Delta_{H,\infty} f}{|\nabla_H f|^2} \right],
\end{equation}
where we have defined
\begin{equation}\label{LapHinfty}
\Delta_{H,\infty} f = \frac 12 \langle\nabla_H(|\nabla_H f|^2),\nabla_H f\rangle.
\end{equation}
In \eqref{gradG}, \eqref{bg2} we have already computed the terms $|\nh f|^2$ and $\Delta_H f$ in \eqref{plaplap}, so we are left with expressing $\Delta_{H,\infty} f$ in terms of the function $g$. In order to proceed let us assume now that
\begin{equation}\label{FG}
F(z,\sigma) = G(\frac{|z|^2}{4},\sigma)
\end{equation}
be another function with partial symmetry. Applying \eqref{Xi} to both $f$ and $F$ it is easy to deduce similarly to the above computation of $|\nh f|^2$ that
\begin{equation}\label{grado2}
\sa\nh F,\nh f\da = \xi (G_\xi g_{\xi} + \sa\nabla_\sigma G,\nabla_{\sigma}g\da).
\end{equation}
If we now take $F = |\nh f|^2$, and denote by
\[
|\nabla g|^2 = |\nabla_{(\xi,\sigma)} g|^2 = g_\xi^2 + |\nabla_\sigma g|^2,
\]
then in view of \eqref{gradG}, \eqref{LapHinfty} and \eqref{grado2} we find
\begin{equation}\label{inftycyl}
\Delta_{H,\infty} f  = \xi^2 |\nabla g|^2\left\{\frac{\Delta_\infty g}{|\nabla g|^2} + \frac{g_{\xi}}{2\xi}\right\},
\end{equation}
where
\[
\Delta_\infty g = \frac 12 \sa\nabla(|\nabla g|^2),\nabla g\da.
\]
Substituting \eqref{gradG}, \eqref{bg2} and \eqref{inftycyl} in \eqref{plaplap}, we finally obtain
\begin{align*}
\Delta_{H,p} f & = |\nabla_H f|^{p-2} \left[\Delta_H f + (p-2) \frac{\Delta_{H,\infty} f}{|\nabla_H f|^2} \right]
\\
& = \xi^{\frac{p-2}2}\ |\nabla g|^{p-2}\left[\xi\ \left(g_{\xi\xi} + \frac{m}{2\xi}\ g_\xi+ \Delta_\sigma  g\right)+ (p-2)\frac{\xi^2 |\nabla g|^2\left\{\frac{\Delta_\infty g}{|\nabla g|^2} + \frac{g_{\xi}}{2\xi}\right\}}{\xi\ |\nabla g|^{2}}\right]
\\
& = \xi^{\frac{p}2}\ |\nabla g|^{p-2}\left[g_{\xi\xi} + \frac{m+p-2}{2\xi}\ g_\xi+ \Delta_\sigma  g + (p-2)\frac{\Delta_\infty g}{|\nabla g|^2}\right].
\end{align*}
We have thus proved \eqref{pisub2}.

\end{proof}

\begin{remark}\label{R:plap}
In connection with \eqref{pisub2} it is important to emphasise that if we think of $(w,\sigma)\in \R^{\frac{m+p}2}\times \R^k$, then from a comparison with \eqref{plapclassic} we reach the important conclusion that the standard $p$-Laplacian acting on a function $H:\R^{\frac{m+p}2}\times \R^k\to \R$ depending on $\xi = |w|$ and $\sigma\in \R^k$, i.e. such that $H(w,\sigma) = h(|w|,\sigma)$, is precisely the operator
\begin{equation}\label{pcylDel}
\Delta_{p} H = |\nabla h|^{p-2}\left[h_{\xi\xi} + \frac{m+p-2}{2\xi}\ h_\xi+ \Delta_\sigma  h + (p-2)\frac{\Delta_\infty h}{|\nabla h|^2}\right].
\end{equation}
\end{remark}

We mention that this phenomenon of \emph{fractional dimensionality} for a nonlinear degenerate operator such as $\Delta_p$ is connected to the appearance in \eqref{pcylDel} of the Bessel process generated by
\[
\p_{\xi\xi} + \frac{m+p-2}{2\xi}\ \p_\xi.
\]
This is akin to what happens in the famous Caffarelli-Silvestre extension procedure where the fractal dimension is $2(1-s)$, with $s\in (0,1)$ being the fractional power of $(-\Delta)^s$, see \cite{CS}.



\section{Proof of Theorem \ref{T:main}}\label{S:main}

In this section we prove Theorem \ref{T:main}. Throughout the section, we will be working under the assumptions on $\Om$ and $f$ in Theorem \ref{T:main}. We begin by observing the following simple lemma.

\begin{lemma}\label{L:simple}
Let $f$ be a solution to the Dirichlet problem \eqref{i3}. Then
\[
\int_{\Om} |\nh f|^p dx = \int_{\Om} f |z|^p dx.
\]
\end{lemma}

\begin{proof}
It suffices to take $\vf = f$ as a test function in \eqref{ws}.

\end{proof}

Next, we show that if $f$ satisfies the overdetermined assumption \eqref{od}, then the constant $c$ must be strictly positive. In what follows we indicate with $d\sigma$ the differential of $(n-1)$-dimensional measure on $\p \Om$. Furthermore, we indicate with $\nabla f$ the standard (Riemannian) gradient as in \eqref{rg} of a function $f$.

\begin{lemma}\label{L:cpos}
Suppose that $f$ be a solution to \eqref{i3}. If $f$ satisfies \eqref{od}, then
\[
\int_{\Om} |z|^p dx = c^p \int_{\p \Om} \frac{|z|^p}{|\nabla f|} d\sigma.
\]
In particular, we must have $c>0$.
\end{lemma}

\begin{proof}
By \cite[Lemma 2.6]{D} we have $f\ge 0$ in $\Om$. The Harnack inequality in \cite{CDG} implies that $f>0$ in $\Om$. This gives $\nu = - \frac{\nabla f}{|\nabla f|}$ on $\p \Om$. Using \eqref{i3}, and integrating by parts, we thus find
\begin{align*}
& \int_{\Om} |z|^p dx = - \int_{\Om} \Delta_{H,p} f dx = - \sum_{i=1}^m \int_{\p \Om} |\nh f|^{p-2} X_i f \langle X_i,\nu\rangle d\sigma
\\
& = \int_{\p \Om} \frac{|\nh f|^p}{|\nabla f|} d\sigma = c^p \int_{\p \Om} \frac{|z|^p}{|\nabla f|} d\sigma,
\end{align*}
as desired. This latter equation implies that it must be $c>0$.

\end{proof}

In what follows, we indicate with $\mathscr Z$ the infinitesimal generator of the group dilations \eqref{dilG}. We will need the following elementary facts established in \cite{DGtrieste}.

\begin{lemma}\label{L:Zprop}
The vector field  $\mathscr Z$ enjoys the following properties:
\begin{itemize}
\item[(i)]
\emph{div}$_\bG \mathscr Z \equiv Q$.
\item[(ii)]
One has $[X_i,\mathscr Z]=X_i, \quad i=1,...,m,$
\item[(iii)]
$\Delta_H(\mathscr Zu) = \mathscr Z(\Delta_H u) + 2 \Delta_Hu$, \emph{for any} $u\in C^{\infty}(\bG)$.
In particular, $\mathscr Zu$ is harmonic if such is $u$.
\end{itemize}
\end{lemma}

Our next result is a generalisation of the Rellich identity first found in \cite{GV}, see also \cite{GR}.

\begin{proposition}\label{P:rellich}
Let $f\in \Gamma^{2,\alpha}(\overline \Om)$ and assume that $\nabla f$ does not vanish on $\p \Om$ and that $\nu = - \frac{\nabla f}{|\nabla f|}$. Then
\[
(p-1) \int_{\p \Om} |\nh f|^p \langle \mathscr Z,\nu\rangle d\sigma = p \int_{\Om} \mathscr Zf\  \Delta_{H,p} f dx - (Q-p) \int_{\Om} |\nh f|^p dx.
\]
\end{proposition}

\begin{proof}
The divergence theorem and (i) of Lemma \ref{L:Zprop} give
\begin{align*}
& \int_{\p \Om} |\nh f|^p \langle \mathscr Z,\nu\rangle d\sigma = \int_{\Om} \operatorname{div}(|\nh f|^p \mathscr Z) dx
\\
& = Q \int_{\Om} |\nh f|^p dx + \int_{\Om} \mathscr Z\left(|\nh f|^p\right) dx
\\
& = Q \int_{\Om} |\nh f|^p dx + p \sum_{i=1}^m \int_{\Om} |\nh f|^{p-2} X_i f \mathscr Z(X_i f) dx
\\
& = Q \int_{\Om} |\nh f|^p dx + p \sum_{i=1}^m \int_{\Om} |\nh f|^{p-2} X_i f X_i(\mathscr Z f) dx - p \sum_{i=1}^m \int_{\Om} |\nh f|^{p-2} X_i f [X_i,\mathscr Z] f dx
\\
& = (Q-p) \int_{\Om} |\nh f|^p dx + p \sum_{i=1}^m \int_{\Om} |\nh f|^{p-2} X_i f X_i(\mathscr Z f) dx,
\end{align*}
where in the last equality we have used (ii) of Lemma \ref{L:Zprop}. We now integrate by parts in the last integral in the right-hand side, obtaining
\begin{align*}
& p \sum_{i=1}^m \int_{\Om} |\nh f|^{p-2} X_i f X_i(\mathscr Z f) dx = p \sum_{i=1}^m \int_{\p \Om} |\nh f|^{p-2} X_i f \mathscr Z f \langle X_i,\nu\rangle d\sigma - p \int_{\Om} \mathscr Zf \Delta_{H,p} f dx
\\
& =  p \sum_{i=1}^m \int_{\p \Om} |\nh f|^{p} \langle \mathscr Z,\nu\rangle d\sigma - p \int_{\Om} \mathscr Zf \Delta_{H,p} f dx.
\end{align*}
Substituting in the above equality, and rearranging terms, we reach the desired conclusion.

\end{proof}

We are now ready to present the

\begin{proof}[Proof of Theorem \ref{T:main}]
Suppose that $f$ be a solution to \eqref{i3} satisfying the overdetermined condition \eqref{od}. By Lemma \ref{L:cpos} we know that $c>0$. This information allows us to get started in analysing the terms in the integral identity in Proposition \ref{P:rellich}. We begin with
\begin{align*}
& p \int_{\Om} \mathscr Zf \Delta_{H,p} f dx = - p  \int_{\Om} |z|^p \mathscr Zf  dx = - \int_{\Om} \mathscr Z(|z|^p f)  dx + p \int_{\Om} f \mathscr Z(|z|^p)  dx
\\
& = - p  \int_{\p \Om} \langle \mathscr Z,\nu\rangle |z|^p f  d\sigma + p Q \int_{\Om} |z|^p f dx + p^2 \int_{\Om} |z|^p f dx
\\
& = p(Q+p) \int_{\Om} |z|^p f dx,
\end{align*}
where in the second to the last equality we have used (i) in Lemma \ref{L:Zprop} and the fact that $\mathscr Z(|z|^p) = p\ |z|^p$, whereas in the last equality we have used the information $f = 0$ on $\p \Om$. Next, we have
\begin{align*}
& (p-1) \int_{\p \Om} |\nh f|^p \langle \mathscr Z,\nu\rangle d\sigma  = (p-1) c^p \int_{\p \Om} |z|^p \langle \mathscr Z,\nu\rangle d\sigma
\\
& = (p-1) c^p \int_{\Om} \operatorname{div}_{\bG} (|z|^p \mathscr Z) dx = (p-1) Q c^p \int_{\Om} |z|^p  dx + (p-1) c^p \int_{\Om} \mathscr Z(|z|^p) dx
\\
& = (p-1)(Q +p) c^p \int_{\Om} |z|^p  dx,
\end{align*}
where again we have used (i) in Lemma \ref{L:Zprop} and that $\mathscr Z(|z|^p) = p |z|^p$. Substituting the latter two identities in Proposition \ref{P:rellich} we find
\begin{align*}
& (p-1)(Q +p) c^p \int_{\Om} |z|^p  dx = p(Q+p) \int_{\Om} |z|^p f dx - (Q-p) \int_{\Om} |\nh f|^p dx.
\end{align*}
At this point we repeatedly use Lemma \ref{L:simple} to critically reorganise the terms in the latter identity in the following way
\begin{align*}
& (p-1)(Q +p) c^p \int_{\Om} |z|^p  dx = (p-1)Q \int_{\Om} |\nh f|^p dx + p \int_{\Om} |\nh f|^p dx + p^2 \int_{\Om} |z|^p f dx
\\
& = (p-1)(Q+p) \int_{\Om} |\nh f|^p dx - p(p-1) \int_{\Om} |\nh f|^p dx
\\
& + p \int_{\Om} |\nh f|^p dx + p^2 \int_{\Om} |z|^p f dx
\\
& = (p-1)(Q+p) \int_{\Om} |\nh f|^p dx + 2 p \int_{\Om} |z|^p f dx.
\end{align*}
If we now divide by $(p-1)(Q +p)$ both sides of the latter equality, we finally reach the desired conclusion \eqref{ii}.

\end{proof}

We close this section by mentioning that the question of the regularity at the boundary of solutions of problems such as \eqref{i3} is one of basic independent interest. The issue is twofold: what happens away from the characteristic set of the relevant domain, and how good is the solution up to such set. For the former aspect, considerable progress in the case $p=2$ has recently been made in the work \cite{BGM}, where a complete Schauder theory has been developed, see also the previous works \cite{BGMcv} and \cite{BCCu}. Still in the linear case, for earlier contributions at the characteristic set the reader should see    
\cite{CG}, \cite{CGN1}, \cite{CGN2}. For results in the case $p\not=2$ the theory is still largely undeveloped. We refer the reader to \cite{GV}, \cite{GP} and \cite{Ny}, and to the references therein. 


\section{Some basic invariances}\label{S:inv}

The objective of this section is to prove Proposition \ref{L:Hlapinv}. The next  lemma will be quite important in what follows.

\begin{lemma}\label{L:gradinv}
Let $\Om\subset \bG$ be a domain with partial symmetry as in Definition \ref{D:psym}, and let $S\in \mathscr S(\mathfrak h)$. Let $F, f:\Om\to \R$ be two functions and set
\[
G(z,\sigma) = F(Sz,\sigma),\ \ \ \ \ \ \ \ g(z,\sigma) = f(Sz,\sigma).
\]
Then
\begin{equation}\label{commnab2}
\sa\nh G(z,\sigma),\nh g(z,\sigma)\da = \sa\nh F(Sz,\sigma),\nh f(Sz,\sigma)\da.
\end{equation}
In particular, when $F = f$ we obtain from \eqref{commnab2}
\begin{equation}\label{commnab}
|\nh g(z,\sigma)|^2 = |\nh f(Sz,\sigma)|^2.
\end{equation}
\end{lemma}

\begin{proof}

We begin by observing that the chain rule gives
\begin{equation}\label{nab}
\nabla g(z,\sigma) = \left(S^t \nabla_z f(Sz,\sigma),\nabla_\sigma f(Sz,\sigma)\right),
\end{equation}
and similar expression for $G$.
To prove \eqref{commnab2} we note that \eqref{Xi} and \eqref{nab} give
\begin{align*}
& \sa\nh G(z,\sigma),\nh g(z,\sigma)\da = \sum_{i=1}^m \left(\p_{z_i}G + \frac 12 \sum_{\ell=1}^k \sa J(\ve_\ell)z,e_i\da \partial_{\sigma_\ell}G\right)\left(\p_{z_i}g + \frac 12 \sum_{\ell'=1}^k \sa J(\ve_\ell')z,e_i\da \partial_{\sigma_{\ell'}}g\right)
\\
& = \sa \nabla_z G,\nabla_z g\da + \frac 12 \sum_{\ell=1}^k \sum_{i=1}^m \sa\nabla_z G,e_i\da \sa J(\ve_\ell)z,e_i\da \partial_{\sigma_\ell}g
 +  \frac 12 \sum_{\ell=1}^k \sum_{i=1}^m \sa\nabla_z g,e_i\da \sa J(\ve_\ell)z,e_i\da \partial_{\sigma_\ell} G
 \\
 & + \frac 14 \sum_{\ell, \ell'=1}^k  \sum_{i=1}^m  \sa J(\ve_\ell)z,e_i\da \sa J(\ve_\ell')z,e_i\da \partial_{\sigma_\ell}G  \partial_{\sigma_{\ell'}}g
 \\
& = \sa S^t \nabla_z F(Sz,\sigma),S^t \nabla_z f(Sz,\sigma)\da+ \frac 12 \sum_{\ell=1}^k \sum_{i=1}^m \sa S^t \nabla_z F(Sz,\sigma),e_i\da \sa J(\ve_\ell)z,e_i\da \partial_{\sigma_\ell}f(Sz,\sigma)
\\
& +  \frac 12 \sum_{\ell=1}^k \sum_{i=1}^m \sa S^t \nabla_z f(Sz,\sigma),e_i\da \sa J(\ve_\ell)z,e_i\da \partial_{\sigma_\ell} g(Sz,\sigma)
\\
& = \sa S S^t \nabla_z F(Sz,\sigma),\nabla_z f(Sz,\sigma)\da + \frac 12 \sum_{\ell=1}^k \sa J(\ve_\ell)z,S^t \nabla_z F(Sz,\sigma)\da \partial_{\sigma_\ell}f(Sz,\sigma)
\\
& + \frac 12 \sum_{\ell=1}^k \sa J(\ve_\ell)z,S^t \nabla_z f(Sz,\sigma)\da \partial_{\sigma_\ell}F(Sz,\sigma)  + \frac{|z|^2}4 \sum_{\ell, \ell'=1}^k  \delta_{\ell\ell'} \partial_{\sigma_\ell}G  \partial_{\sigma_{\ell'}}g
\\
& = \sa \nabla_z F(Sz,\sigma),\nabla_z f(Sz,\sigma)\da + \frac 12 \sum_{\ell=1}^k \sa S J(\ve_\ell)z,\nabla_z F(Sz,\sigma)\da \partial_{\sigma_\ell}f(Sz,\sigma)
\\
& + \frac 12 \sum_{\ell=1}^k \sa S J(\ve_\ell)z,\nabla_z f(Sz,\sigma)\da \partial_{\sigma_\ell}F(Sz,\sigma) + \frac{|Sz|^2}4 \sa\partial_{\sigma} G,\partial_{\sigma} g\da,
\end{align*}
where we have used that $S^t S = I$, that $|Sz| = |z|$, and also \eqref{ip}.
From the right-hand side of the last equation and from \eqref{commell}, we now infer
\begin{align*}
& \sa\nh G(z,\sigma),\nh g(z,\sigma)\da = \sa \nabla_z F(Sz,\sigma),\nabla_z f(Sz,\sigma)\da + \frac 12 \sum_{\ell=1}^k \sa J(\ve_\ell) Sz,\nabla_z F(Sz,\sigma)\da \partial_{\sigma_\ell}f(Sz,\sigma)
\\
& + \frac 12 \sum_{\ell=1}^k \sa J(\ve_\ell) Sz,\nabla_z f(Sz,\sigma)\da \partial_{\sigma_\ell}F(Sz,\sigma) + \frac{|Sz|^2}4 \sa\partial_{\sigma} F(Sz,\sigma),\partial_{\sigma} f(Sz,\sigma)\da.
\end{align*}
This proves \eqref{commnab2}.

\end{proof}

\begin{proof}[Proof of Proposition \ref{L:Hlapinv}]
We apply again \eqref{Xi}, obtaining
\begin{align*}
\sul g(z,\sigma) & = \Delta_z g + \frac 12 \sum_{\ell=1}^k \sum_{i=1}^m \big\{\p_{z_i}(\sa J(\ve_\ell)z,e_i\da) \p_{\sigma_\ell} g + \sa J(\ve_\ell)z,e_i\da \p_{z_i}\big(\p_{\sigma_\ell}  g\big)\big\}
\\
& + \frac 12 \sum_{\ell=1}^k \sum_{i=1}^m \sa J(\ve_\ell)z,e_i\da + \frac{|z|^2}4 \Delta_\sigma g
\\
& = \Delta_z f(Sz,\sigma) + \frac{|Sz|^2}4 \Delta_\sigma f(Sz,\sigma) + \sum_{\ell=1}^k \sum_{i=1}^m \sa J(\ve_\ell)z,e_i\da \sa \nabla_z \big(\p_{\sigma_\ell} g\big),e_i\da
\\
& = \Delta_z f(Sz,\sigma) + \frac{|Sz|^2}4 \Delta_\sigma f(Sz,\sigma) + \sum_{\ell=1}^k \sa J(\ve_\ell)z,\nabla_z \big(\p_{\sigma_\ell} g\big)\da
\\
& = \Delta_z f(Sz,\sigma) + \frac{|Sz|^2}4 \Delta_\sigma f(Sz,\sigma) + \sum_{\ell=1}^k \sa J(\ve_\ell)z,S^t \nabla_z \big(\p_{\sigma_\ell} f\big)(Sz,\sigma)\da
\\
& = \Delta_z f(Sz,\sigma) + \frac{|Sz|^2}4 \Delta_\sigma f(Sz,\sigma) + \sum_{\ell=1}^k \sa J(\ve_\ell) Sz,\nabla_z \big(\p_{\sigma_\ell} f\big)(Sz,\sigma)\da
\\
& = \sul f(Sz,\sigma).
\end{align*}
This proves \eqref{invlap}. We note explicitly that in passing from the first to the second   equality we have used the following fact. If we write $z = \sum_{j=1}^m z_j e_j$, then
\[
\p_{z_i}(\sa J(\ve_\ell)z,e_i\da) = \sum_{j=1}^m \p_{z_i}(z_j \sa J(\ve_\ell) e_j,e_i\da) = \sum_{j=1}^m \delta_{ij} \sa J(\ve_\ell) e_j,e_i\da = \sa J(\ve_\ell) e_i,e_i\da = 0.
\]
To establish \eqref{invlapinfty}, we first observe that if we let $F(z,\sigma) = |\nh f(z,\sigma)|^2$, then \eqref{commnab} implies that
\[
G(z,\sigma) = F(Sz,\sigma) = |\nh f(Sz,\sigma)|^2 = |\nh g(z,\sigma)|^2.
\]
From this observation and from \eqref{LapHinfty}, applying \eqref{commnab2} we thus obtain
\begin{align*}
\Delta_{H,\infty} g(z,\sigma)  = \frac 12 \sa \nh G(z,\sigma),\nh g(z,\sigma)\da = \frac 12 \sa \nh(|\nh f|^2)(Sz,\sigma),\nh f(Sz,\sigma)\da = \Delta_{H,\infty} f(Sz,\sigma),
\end{align*}
which proves \eqref{invlapinfty}. Finally, the equation \eqref{invplap} follows from an application to the function $g(z,\sigma)=f(Sz,\sigma)$ of the identities \eqref{commnab}, \eqref{invlap} and \eqref{invlapinfty}.

\end{proof}

\vskip 0.2in


\section{The transitivity Property (H)}\label{S:cyl}

Our goal in this section is to provide a characterisation of those group $\bG$ of Heisenberg type for which the  $J-$\emph{invariant} orthogonal group $\mathscr S(\mathfrak h)$ acts transitively on the unit sphere in $\Vone$. While such result is instrumental to those in the next section, it is also of independent interest. As Theorem \ref{t:property H} shows, Property (H) is satisfied in many geometric situations of interest for this paper.

\begin{theorem}\label{t:property H}
A group of Heisenberg type satisfies \emph{Property (H)} if and only if it is isomorphic to a (real) Heisenberg group, or a quaternionic Heisenberg group, or a real version of a complexified Heisenberg group.
\end{theorem}

The proof of Theorem \ref{t:property H} will be presented at the end of this section. We begin the analysis by deriving necessary conditions for Property (H) to hold true. For this we will first identify $\mathscr S(\mathfrak h)$ as a subgroup of some larger groups of automorphisms.  Let $A(\bG)$  be the group of isometric automorphisms of $\bG$. Since $\bG$ is simply connected $A(\bG)$ is isomorphic to the group $A(\algg)$ of isometric automorphisms of $\algg$, see for example \cite[Theorem 1]{Hoch52}.

Although we will not explicitly use this fact, we mention that the group of isometries of $\bG$ is the semidirect product $A(\bG) \ltimes \bG$ with $\bG$  acting by left translations, see \cite[Theorem p. 133]{Ka2}.
A further study and classification of  $A(\algg)$  was carried in \cite{Riehm82}. The classification of the group of automorphisms (not necessarily isometries) was carried in \cite{Saal96}.  Since $A(\algg)$ is a group of isometric automorphisms it preserves the inner product and the bracket, hence we must have $ A(\algg)\subset O(\Vone)\times O(\Vtwo)$. Indeed,  if $\Phi\in  A(\algg)$ and  $\sigma, \, \sigma'\in\Vtwo$ and $z\in\Vone$ it follows that $\text{ad}_{\Phi(\sigma)}=0$ since   we have $[\Phi(\sigma),z]= \Phi([\sigma, \Phi^{-1}(z)])=0$ and  $[\Phi(\sigma),\sigma']=0$, hence $\Phi(\Vtwo)=\Vtwo$  and   then the claim, recalling that $\Vone$ and $\Vtwo$ are orthogonal to each other. In particular, we have the following result stated in \cite[p.135]{Ka2}.

\begin{proposition}\label{p:aut iso}
Two orthogonal transformations $S_1\in O(\Vone)$ and $S_2\in O(\Vtwo)$ induce a Lie algebra automorphism $\Phi=S_1\oplus S_2$ on $\bG$, i.e., $\Phi (z,\sigma)= (S_1 z,S_2\sigma)$, $z\in\Vone$, $\sigma\in \Vtwo$, if and only if for all $\sigma\in \Vtwo$ we have the intertwining relation
\begin{equation}\label{e:sum of isom}
S_1\circ J(\sigma) =J(S_2\, \sigma)\circ S_1.
\end{equation}
In particular, the $J-$invariant orthogonal group  $\mathscr S(\mathfrak h)$ of $\Vone$ is the subgroup of $O(\Vone)$ which fixes the center $\Vtwo$ pointwise, $S_2=I_{\Vtwo}$.
\end{proposition}

\begin{proof}
Indeed, with $z, z'\in\Vone$ and $\sigma, \sigma'\in\Vtwo$  from \eqref{BCH} and the fact that $\Phi$ is an automorphism, the equation
\begin{equation}\label{e:automorphism}
\Phi\left[(z,\sigma),(z',\sigma')\right]=\left[\Phi \left((z,\sigma)\right),\Phi\left((z',\sigma')\right)\right]
\end{equation}
is equivalent to
\[
 S_2[z,z']= [S_1z,S_1z'].
\]
Thus, for any $\sigma\in \Vtwo$ we have from \eqref{kap}
\[
\sa J(S_2^t\sigma)z, z'\da = \sa S_2^t\sigma,[z,z']\ds = \sa \sigma,S_2[z,z']\da =  \sa \sigma,[S_1z,S_1z']\da = \sa J(\sigma)S_1z, S_1z'\da,
\]
which gives \eqref{e:sum of isom} after replacing $\sigma$ with $S_2\sigma$ and using that the maps $S_j$ are orthogonal.

The last claim follows from   \eqref{e:sum of isom}, noting that $\mathscr S(\mathfrak h)$  is exactly the elements of $ O(\Vone)$ which commute with
all  of the almost complex structures  $J(\sigma):\Vone\to \Vone$, $|\sigma|=1$ and $S_2=I_{\Vtwo}$.

\end{proof}

As a side remark, we mention that one can also identify $\mathscr S(\mathfrak h)$ with the automorphism group of the associated Clifford module, \cite{KR07}.
In regards to Property (H),  note that in the case of Iwasawa groups $A(\algg)$ is transitive on the product of the two unit spheres in $\Vone\times\Vtwo$ by the Kostant double transitivity theorem, \cite[Theorem 6.2]{CDKR2}. However, this is not true in general, which leads to the following Proposition giving \emph{necessary conditions} for Property (H) to hold true. In particular,  using also Proposition \ref{p:aut iso}, the fact that there exist groups of Heisenberg type with center of arbitrary dimension, and noting that $\mathscr S(\mathfrak h)\subset A_{\Vone}(\algg)$, where $A_{\Vone}(\algg)$ denotes the restriction of $A(\algg)$ to $\Vone$,   not every Heisenberg type group has the Property (H) of Definition \ref{D:trans}.

For clarifying the statement of the next proposition, we recall that \emph{isotypic} module means that  all irreducible sub-modules are equivalent, see \cite{KR07}, noting that, correspondingly, an H-type algebra is \emph{reducible} if there exists  a proper decomposition $$\algg=\Vone_1\oplus\Vone_2\oplus\Vtwo,$$ such that $\Vone_1\oplus\Vtwo$ and $\Vone_2\oplus\Vtwo$ are both H-type algebras  with respect to the inherited structures. An H-type algebra is \emph{irreducible} if it is not reducible. The paper \cite{Barbano98} contains an explicit description of all irreducible H-type groups with centers of dimension $0\leq k\leq 7$, based on the composition of quadratic forms.

\begin{proposition}\label{p:1st necessary condition}
In a group of Heisenberg type, if Property $(H)$ holds then we have
 \begin{equation}\label{e:general case H prop}
k=\dim\Vtwo\in\{1,2,3,5,6,7\},\  \Vone=\Vone_0^l,\quad 1\leq l\in \mathbb{Z},\  \quad \text{ with $\ l=1\ $ if  $\ k\in\{5,6,7\}$}.
 \end{equation}
Furthermore,  the linear sub-space $\Vone_0$  has the following dimension depending on the dimension of the center:
\begin{equation}\label{e:irr H prop}
\begin{array}{c|c|c|c|c|c|c}
  k & 1 & 2 & 3 & 5 & 6 & 7 \\
  \hline
  \dim\Vone_0 & 2 & 4 & 4 & 8 & 8 & 8
\end{array}.
\end{equation}
\end{proposition}

\begin{proof}
  For any H-type group, from \cite{Riehm82}, see also \cite[p. 199]{K84}, we have a complete characterization when $A_{\Vone}(\algg)$ is transitive on the unit sphere in $\Vone$:
\begin{enumerate}[i)]
\item if the dimension of the center $k=1$ or $k=2$ then $A_{\Vone}(\algg)$ is always transitive on the sphere in the first layer;
\item if $k\in\{5,6,7\}$ then $A_{\Vone}(\algg)$ is transitive on the sphere in $\Vone$  if and only if $m=\dim\Vone=8$;
\item if $k=3$ , $A_{\Vone}(\algg)$ is transitive on the sphere in $\Vone$ if and only if  $\Vone$ is isotypic as a module over the Clifford algebra Cliff$(\Vone)$ of the quadratic form $-|\xi_1|^2$, $\xi_1\in \Vone$;
    \item   $A_{\Vone}(\algg)$ is not transitive whenever $\dim \Vtwo=4$ or $\dim\Vtwo\geq 8$.
\end{enumerate}

Suppose $\bG$ satisfy the Property (H).
From (i) -- (iv)  above and $\mathscr S(\mathfrak h)\subset A_{\Vone}(\algg)$, the following conditions must hold true:
\begin{equation}\label{e:irreduc}
\begin{aligned}
 & (i) \quad 1\leq k\leq 7, \quad k\not=4;\\  
 & (ii) \quad \text{if $k=3$ then $\Vone$ is isotipic};\\ \ 
 & (iii) \quad \text{if $k\in\{5,6,7\}$ then $m=\dim\Vone=8$}.
\end{aligned}
\end{equation}
From the classification of the H-type Lie algebras, see \cite{Ka2} and \cite[pp.119-121]{K84a}, it follows that the groups of Heisenberg type satisfying the conditions (i)-(iii) with centers of dimension 1, 3 or 7 are exactly the non-Abelian Iwasawa type groups, while when the center is two-dimensional the group is the real version of the complex Heisenberg group. We will study these four cases explicitly later in the section. Let us mention explicitly the  irreducible cases satisfying the necessary conditions \eqref{e:irreduc}.  If $k\in\{5,6,7\}$ then $\Vone=\Vone_0$.
The cases $k=1,3,7$ with $l=1$ are the irreducible Iwasawa type groups of \emph{lowest} possible dimensions. The general cases, including the irreducible ones,  of the Iwasawa groups are considered in Proposition \ref{P:iwa}. The case $k=2$ is the lowest dimensional real version of the complexified Heisenberg group, see Example \ref{eg:complex Heisnebrg} and  Proposition \ref{p:complex Heisnebrg} for the reducible case.

On the other hand, by  \cite[p. 419]{KapR} the non Abelian \emph{irreducible} $H$-type algebras satisfying \eqref{e:irreduc} have the dimensions of their first layer  $\Vone_0$ specified in \eqref{e:irr H prop}. By the definition of reducible H-type algebra and (i) - (iii), it follows that $\Vone$ has the claimed in \eqref{e:general case H prop} form.

\end{proof}

Next, we will show that Property (H) does not hold if the center is of dimension 5, 6, or 7.

\begin{proposition}\label{p:2nd necessary conds}
 If $\bG$ is a group of Heisenberg type which has \emph{Property (H)}, then (up to isomorphism) $\bG$ is one of the real or quaternionic Heisenberg groups, or the real version of a complex Heisenberg group.
 \end{proposition}

\begin{proof}

In general, the group $\mathscr S(\mathfrak h)$ has been classified up to an isomorphism in \cite[p. 65]{Saal09}. 
Specializing \cite[p. 65]{Saal09} to the case when the necessary condition \eqref{e:general case H prop} holds true, then the group $\mathscr S(\mathfrak h)$ is algebraically isomorphic to one of the following groups, where $k=\dim\Vtwo$ denotes the dimension of the center $\Vtwo$ and $m=\dim\Vone$ denotes the dimension of the horizontal space:
\begin{enumerate}[i)]
 \item the complex unitary group $U(m/2)$ if $k= 1$;
  \item  the quaternion unitary group $Sp(m/4)$ if $k= 2$; 
  \item  $Sp(m/4)$ if $k= 3$;
  \item  the orthogonal group $O(2)$ if $k= 5$;
  \item $\mathbb{Z}_2$ if $k= 6,7$. 
\end{enumerate}
As a consequence, Property (H) can hold only in the first three cases taking into account that $k=\dim \Vone=8$ in the last two cases by Proposition \ref{p:1st necessary condition}.
\end{proof}

Before turning to the proof of Proposition \ref{P:iwa} in which we show that  the Iwasawa type groups excluding  the octonion case have property (H) we make a few comments. Recall that the Iwasawa type groups  are the nilpotent part in the Iwasawa decomposition of the simple Lie groups of real rank one, hence   the isometry group of the non-compact symmetric spaces  of real rank one. Thus, the Iwasawa type groups can be described as the boundaries of the  real, complex, quaternionic or octonionic hyperbolic space. Alternatively, they can be characterized as the Heisenberg type groups that satisfy the $J^2$-condition, \cite[Theorem 1.1]{CDKR}, \cite[Theorem 4.2]{CDKR}, \cite[Theorem 6.1]{CDKR}, \cite{CDKR2} and \cite{KR07}.  Recall  that the Heisenberg type algebra $\algn = \mathfrak h \oplus \mathfrak v$ satisfies the  $J^2$-condition if for any $z\in \mathfrak h$ and every $\sigma, \sigma'\in \mathfrak v$ such that $\sa\sigma, \sigma'\da=0$, there exists $\sigma''\in \mathfrak v$ such that
\begin{equation}\label{e:J2 cond}
J(\sigma)J(\sigma')z=J(\sigma'')z.
\end{equation}

\begin{proof}[Proof of Proposition \ref{P:iwa}]

The proof of the transitivity of  $\mathscr S(\mathfrak h)$ on the sphere in $\Vone$ in the case of the complex and quaternionic Heisenberg groups is contained in the proof of the Kostant double transitivity theorem presented in \cite[Theorem 6.2]{CDKR2}. Note that $Ad(M)$ used in \cite[Theorem 6.2]{CDKR2} is the group $A(\bG)$ here taking into account \cite[Lemma 5.5]{CDKR2}. Alternatively, a proof is contained in the proof of \cite[Proposition 2.2]{KR10}. That transitivity does not hold in the case of the octonian Heisenberg group follows from Proposition \ref{p:2nd necessary conds}.

Below, we include alternative elementary proofs  that will use the explicit form   of the Iwasawa groups given in \cite[Section 9.3]{Pansu89} as boundaries of the respective hyperbolic spaces \cite[Section 19]{Mo}.
Let $\mathbb{K}$ denotes one of the real division algebras: the real numbers $\mathbb{R}$, the complex numbers $\mathbb{C}$, the quaternions $\mathbb{H}$, or the octonions $\mathbb{O}$ and $\bar q$ denotes the respective conjugation.  We will use the obvious multidimensional version of the following identifications in the one dimensional case:  in the complex case $q=x+iy \in \mathbb{C}$ is identified with $(x,y)\in {\R}^{2}$; $q=t+ix+jy+kz \in \mathbb{H}$ is identified with $(t,x,y,z)\in {\R}^4$, and $q=\sum_{\alpha=0}^7 t_\alpha \textbf{e}_\alpha\in \mathbb{O}$ is identified with $(t_0,\dots,t_7)\in \R^8$, where  $\textbf{e}_\alpha$, $\alpha=1,\dots,7$, are the standard unit imaginary octonions and  $ \textbf{e}_0=1$, see \cite{Baez02}.  The Iwasawa group is isometrically isomorphic to $\mathbb{R}^n$ in the degenerate case when the Iwasawa group is Abelian or to one of the Heisenberg groups  $\bG=\QK \ =\ \Kn\times\text {Im}\, \mathbb{K}$ \index{Heisenberg group $\QK$} with the group law
given by
\begin{equation}\label{e:H-type Iwasawa groups}
  (q_o, \omega_o)\circ(q, \omega)\ =\ \left(q_o\ +\ q, \omega\ +\ \omega_o\ + \ \frac 12 \text
{Im}\  (\bar q_o\, q)\right),
\end{equation}
where $q,\ q_o\in\Kn$, $\bar q_o\,  q=\sum_{\alpha=1}^n \bar q_o^\alpha\,  q^\alpha$, with $q_o^\alpha$ and $  q^\alpha$ respectively denoting complex, quaternionic or octonion coordinates of $q_0$ and $q$, and $\omega, \omega_o\in \text {Im}\, \mathbb{K}$, see also  \cite{IV15}, \cite{CaChMa09}, \cite{IV2}, \cite{WW14} for more details in the particular  cases.
From \eqref{BCH} the  bracket on the Lie algebra is given by
\begin{equation}\label{e:Iwasawa bracket}
[(q_o, \omega_o), \, (q, \omega)]=\text {Im}\  (\bar q_o\, q).
\end{equation}
 Clearly $\text {Im}\, \mathbb{K}$ is the center of the Lie algebra. In the non-Euclidean case the Lie algebra $\algg$ of $\bG$ has center of dimension $1$, $3$, or $7$.  The inner product on $\Kn\times\text {Im}\, \mathbb{K}$ is the standard Euclidean product in the corresponding real vector space given by
\[
\sa(q_o, \omega_o), \, (q, \omega)\da= \text {Re}\  (\bar q_o\, q)+ \text {Re}\  (\bar \omega_o\, \omega).
\]
We will denote by $\Vone$ be the real Euclidean vector space underlining the space $\Kn$, while $\Vtwo$ will denote the real Euclidean vector space defined by $\text {Im}\, \mathbb{K}$.

For a purely imaginary $\omega\in \text {Im}\, \mathbb{K}$ the corresponding map $J_\omega$ on the horizontal space is given by multiplication on the right by $ \omega$ in $\Kn$, i.e., the identity $$\sa\omega, \text {Im}\  (\bar q_o\, q)\da=\sa J_\omega q_0, q\da$$ is equivalent to
\begin{equation}\label{e:form of J}
\sa \omega, \text {Im}\  (\bar q_o\, q)\da\ =\  \sa q_0\omega, q\da.
\end{equation}
To see the latter identity, compute the left-hand side as follows
\begin{multline}\label{e:im part}
\sa\omega, \text {Im}\  (\bar q_o\, q)\da=\frac 12\left[\sa\omega,\bar q_o\, q\da-  \sa\omega,\bar q\, q_0\da \right]
=\frac 12 \left[\sa\bar\omega,\overline {\bar q_o\, q}\da -  \sa\omega,\bar q\, q_0\da \right]\\
=\frac 12 \left[- \sa\omega, {\bar q\, q_0}\da-  \sa\omega,\bar q\, q_0\da\right]=- \sa\omega, {\bar q\, q_0}\da,
\end{multline}
noting that conjugation preserves the inner product and that $\omega\in \mathbb{K}$ means that $\bar\omega=-\omega$. On the other hand,
\begin{equation}\label{e:right multipl}
\sa q_0\omega, q\da =\sa q_0, q\bar \omega\da=\sa\bar q q_0, \bar \omega\da=- \sa\omega, {\bar q\, q_0}\da,
\end{equation}
see for example \cite[Section 1]{Bryant82}.

The maps $J_\omega$ generate a subgroup of SO$(\Vone)$ as $\omega$ runs through the unit sphere of $\text {Im}\, \mathbb{K}$, i.e., of $\Vtwo$. In the  complex case the group is $\mathbb{Z}_2$ generated by one fixed complex structure $J$, $J^2=-I$. In the quaternionic case, the group is (by definition) the group Sp$(1)$. Finally, in the octonionic case, the group is (by definition) the group Spin$(7)$.

Therefore, an orthogonal $\R$-linear map $S$ on $\Kn$ commutes with all $J_\omega$'s,  $S\in\mathscr S(\mathfrak h)$, if and only if $S$ is in the centralizer in SO$(\Vone)$ of the subgroups described above. Hence, $S\in U(n)$ in the complex case; $S\in Sp(n)$, the quaternionic unitary group (acting on the left), in the quaternionic case with $n\geq 1$. Each of these groups is transitive on the unit sphere in $\Kn$ (with respect to its standard linear action), which has been used in the holonomy theorem, see for example \cite{Be}, or can be proven easily directly by mapping adapted orthonormal bases to each other. For example, in the seven dimensional quaternionic group, the  matrices  $S\in\mathscr S(\mathfrak h)$ are given by
\begin{equation}\label{e:l_q}
S=\left(
    \begin{array}{cccc}
      d & -a & -b & -c \\
      a & d & -c & b \\
      b & c & d & -a \\
      a & -b & a & d \\
    \end{array}
  \right), \qquad a^2+b^2+c^2+d^2=1.
\end{equation}
Another short calculation shows that $S$ is the real expression of quaternionic multiplication on the left by a unit quaternion $q=d+ai+bj+ck$. The transitivity is then obvious, because given two unit quaternions $q$ and $p$ we have trivially, $(p\bar q)q=p$.

On the other hand, in the octonionic case, using the non-associativity, the right (resp. left) multiplications can be realized as left (resp. left) multiplications by a unit octonians, hence the right multiplications by unit octonians generate the whole group $SO(8)$, see \cite[Section 8.4, Theorem 7]{CoSm} or \cite[Section B]{MaSc93}. However, in our case we allow only right multiplications by unit purely imaginary octonians, i.e., we have the group Spin(7) $\subset$ SO(8) as the subgroup  generated by the (real form of) right multiplication by unit octonians which are purely imaginary. In this case, the centralizer of Spin(7) $\subset$ SO(8) is $Z_2=\{\pm I\}$, which was pointed to us by R. Bryant. Therefore the transitivity Property (H) in Definition \ref{D:trans} fails.

\end{proof}

In the following example we consider the case of the lowest dimensional  group of Heisenberg type with two dimensional center, see \cite[p.41]{Ka83} and \cite[Section 2]{Rei01}.
\begin{example}\label{eg:complex Heisnebrg}
Consider the complexified Heisenberg group with complex dimension three,  defined by the Lie algebra  (over $\C$) with generators $X_1, \, X_2, \, Z$ with the only non-trivial commutator given by
\[
[X_1,X_2]=Z.
\]
Let $J$ be the complex structure given by multiplication by $i$. As a real Lie algebra this is an algebra of Heisenberg type with $\Vone=\text{span}\, \{X_1, X_2, X_3=JX_1, X_4=JX_2\}$ and center  $\Vtwo=\text{span}\{Z_1(=Z), Z_2=J Z_1 \}$, where the vectors are declared to be an orthonormal basis. The nontrivial commutators are
\[
Z_1=[X_1, X_2]=-[X_3, X_4], \qquad Z_2=[X_1,X_4]=-[X_2,X_3].
\]

Letting $J_j=J_{Z_j}$, $j=1,2$,  \eqref{kap} shows
\[
J_1X_1=X_2, \ J_1X_3=-X_4, \ J_2X_1=X_4, \ J_2X_2=-X_3, \ J_j^2=-I, \ J_{Z_1}J_{Z_2}=-J_{Z_2}J_{Z_1}=J,
\]
hence the H-type property holds.  A small calculations shows that $ S(\mathfrak h)$ is given by the matrices in \eqref{e:l_q}, hence \emph{Property (H)} in Definition \ref{D:trans} is satisfied.
\end{example}

The above example essentially gives the general case of groups of Heisenberg type with center of dimension two. As shown in \cite[Theorem 3.1]{El18} the complex Heisenberg algebras  are the only complex algebras of Heisenberg type as defined in \cite[Definition 2.2]{El18}.

\begin{proposition}\label{p:complex Heisnebrg}
The real form of the complexified Heisenberg group satisfies the \emph{Property (H)}.
\end{proposition}

\begin{proof}
Consider the complex Lie algebra $\algg^c$ with a basis given by $\{W_1, \dots, W_n, W_{n+1}, \dots, W_{2n}, Z \}$, with the only nontrivial commutators given by
\[
[W_j,W_{n+j}]=Z.
\]
The real version is the algebra of Heisenberg type, of real dimension $4n+2$ which has a two-dimensional center. Let $J$ be the real operator corresponding to multiplication by $i$. The Lie bracket is $\C$-bilinear with respect to $J$,  hence, for $j=1,\dots, n$ we have
\begin{equation}\label{e:J and ad}
[W_j, W_{n+j}]=-[JW_j, JW_{n+j}]=Z \qquad \text{and}\qquad [JW_j, W_{n+j}]=[W_j, JW_{n+j}]=JZ.
\end{equation}
To obtain the real version of $\algg^c$ we consider the orthonormal basis of left-invariant vector fields given by
\begin{equation}\label{e:cmpl Heisenberg general}
\begin{aligned}
X_j\, (=W_j), \quad X_{j+1} \, (=W_{j+1}),& \quad X_{j+2} \, (=JW_{j+1}),\quad X_{j+3} \, (=JW_{j+2}), \quad j=1,\dots, n,\\
&  Z_1\, (=Z)\  \text{ and }\  Z_2\, (=JZ),
\end{aligned}
\end{equation}
with non-trivial brackets given according to \eqref{e:J and ad}, i.e., for $j=1,\dots, n$ we have
\begin{equation}\label{e:cmpl Heisenberg general commutators}
[X_j,X_{j+1} ]=-[X_{j+2} , X_{j+3}]=Z_1 \qquad\text{ and }\qquad [X_{j+2},X_{j+1} ]=[X_{j} , X_{j+3}]=Z_2.
\end{equation}
Letting $J_j=J_{Z_j}$, $j=1,2$, \eqref{J2} shows
\[
J_1X_j=X_{j+1}, \ J_1X_{j+2}=-X_{j+3}, \ J_2X_j=X_{j+3}, \ J_2X_{j+1}=-X_{j+2},
\]
while $\ J_j^2=-I, \ J_{Z_1}J_{Z_2}=-J_{Z_2}J_{Z_1}=J,$ hence the H-type property holds.  It is easy to see that after reordering the basis as follows,
\begin{multline}
Y_1=X_1, \, Y_2=X_{n+1}, \, Y_3=X_{2n+1}, \, Y_4=X_{3n+1}, \dots,\\ Y_{4n-3}=X_n,
 \, Y_{4n-2}=X_{2n}, \, Y_{4n-1}=X_{3n}, \, Y_{4n}=X_{4n},
\end{multline}
the complex structures $J_1$ and $J_2$ are given by $4n\times 4n$ real matrices
\[
J_1=\text{diag}(I_1, \dots, I_1),\quad  I_1=\left(
                                              \begin{array}{cccc}
                                                0   & -1 & 0 & 0 \\
                                                1 & 0& 0 & 0 \\
                                                0 & 0 & 0 & 1 \\
                                                0 & 0 & -1 & 0 \\
                                              \end{array}
                                            \right)
\]
and
\[
J_2=\text{diag}(I_2, \dots, I_2),\quad  I_2=\left(
                                              \begin{array}{cccc}
                                                0   & 0 & 0 & -1 \\
                                                0   & 0 & 1 & 0 \\
                                                0 & -1& 0 & 0 \\
                                                1 & 0 & 0 & 0 \\
                                              \end{array}
                                            \right).
\]

A small calculations shows that a $4n\times 4n$ real matrix $ A=(A_{ij})_{i,j=1}^n\in S(\mathfrak h)$, $A_{ij}\in \mathfrak{g}\mathfrak{l}(4,\R)$ if and only if $[I_k,A_{ij}]=0$. Switching back to the $X$ basis we see that $S$  is given by the matrices in the form
\begin{equation}\label{e:L_q}
S=\left(
    \begin{array}{cccc}
      D & -A & -B & -C \\
      A & D & -C & B \\
      B & C & D & -A \\
      A & -B & A & D\\
    \end{array}
  \right)\in O(4n).
\end{equation}
This is exactly the real form of left multiplication in $\Quat^n$ by the unitary matrix, $\bar U^t U=I$,
\[
U=D+Ai+Bj+Ck.
\]
This is the definition of the group $\text{Sp}(n)$, acting transitively on the $4n-1$ dimensional sphere, hence property (H) \ref{D:trans} holds true.

\end{proof}

We conclude the section with the

 \begin{proof}[Proof of Theorem \ref{t:property H}]
 It follows from Proposition \ref{p:2nd necessary conds}, Proposition \ref{P:iwa} and Proposition \ref{p:complex Heisnebrg}.

 \end{proof}

\section{Overdetermination and partial symmetry}\label{S:ops}

The objective of this section is to prove Theorem \ref{T:glps}, Proposition \ref{P:iips} and Theorem \ref{T:gsps}.
 We begin with the following partial solution of Problem 1.

\begin{proof}[Proof of Theorem \ref{T:glps}]
Suppose that $\Om\subset\bG$ have partial symmetry according to Definition \ref{D:psym}, and that $f$ be a solution to \eqref{i3}, \eqref{od}. Let $\Om_+^\star$ be as in Definition \ref{D:psym}. Consider $S\in \mathscr S(\mathfrak h)$ and define $g(z,\sigma) = f(Sz,\sigma)$. Since $\Om$ is invariant under the action of $\mathscr S(\mathfrak h)$, invoking \eqref{commnab} in Lemma \ref{L:gradinv} and \eqref{invplap} in Lemma \ref{L:Hlapinv}, and keeping in mind that $|Sz| = |z|$, we see that $g$ also satisfies problem \eqref{i3}, \eqref{od} in $\Om$. By the comparison principle in \cite[Lemma 2.6]{D} we infer that $g \equiv f$ in $\Om$. This means that for every $(z,\sigma)\in \Om$ and every $S\in \mathscr S(\mathfrak h)$
\[
f(Sz,\sigma) = f(z,\sigma).
\]
We now claim that by the Property (H), $f$ must be of the form \eqref{fgG}. To see this, pick two points $w, z$ on the unit sphere in $\mathfrak h$. By the assumption that $\bG$ satisfies the  Property (H), there exists $S\in \mathscr S(\mathfrak h)$ such that $w = S z$. But then, $f(w,\sigma) = f(Sz,\sigma) = f(z,\sigma)$, and therefore $f(z,\sigma) = \hat f(|z|,\sigma)$ for some function $\hat f(\xi,\sigma)$. By taking $g(\xi,\sigma) = \hat f(2\sqrt \xi,\sigma)$ for $\xi\ge 0$, it is clear that $f(z,\sigma) = g(\frac{|z|^2}4,\sigma)$.

If we now set $n = \frac{m+ p}2 + k$ (never mind the fact that, when $p\not=2$, the dimension $n$ might be fractal), we consider the connected bounded open set $\Om^{\star\star}$ in the space $\Rn$, with variables $(w,\sigma)$, where $w\in \R^\frac{m+ p}2$ and $\sigma\in \R^k$, defined as follows
\begin{equation}\label{Ommanina}
\Om^{\star\star} = \{(w,\sigma)\in \R^{\frac{m+p}2}\times\R^k\mid  (|w|,\sigma)\in \Om_+^\star\}.
\end{equation}
Note that for every $(w,\sigma)\in \Om^{\star\star}$ the section $(\R^{\frac{m+p}2}\times\{\sigma\})\cap \Om^{\star\star}$ is spherically symmetric in the variable $w$.
In $\Om^{\star\star}$ we analyse the function
\begin{equation}\label{gG}
G(w,\sigma) \overset{def}{=} g(|w|,\sigma).
\end{equation}
By \eqref{pcylDel} in Remark \ref{R:plap}, we see that the standard $p$-Laplacian of the function $G:\Om^{\star\star}\to \R$ is given by
\[
\Delta_p G = |\nabla g|^{p-2} \left[g_{\xi\xi} + \frac{\frac{m+p}2 -1}{\xi} g_\xi + \Delta_\sigma g + (p-2) \frac{\Delta_\infty g}{|\nabla g|^2}\right].
\]
Since $|z|^p = (|z|^2)^{\frac p2} = (4\xi)^{\frac p2} = 2^p \xi^{\frac p2}$, by Proposition \ref{P:pcylG}, and by the pde in \eqref{i3}, we see that $G$ satisfies in $\Om^{\star\star}$
the problem
\begin{equation}\label{jim}
\begin{cases}
\Delta_p G = - 2^p,
\\
\\
G_{\big|_{\p \Om^{\star\star}}} = 0,\ \ \ \ \ \ |\nabla G|_{\big|_{\p \Om^{\star\star}}} = 2c.
\end{cases}
\end{equation}
Since $\Delta_p(\la G) = \la^{p-1} \Delta_p G$, we deduce from \eqref{jim} that the function $F = 2^{-\frac{p}{p-1}} G$ solves $\Delta_p F = -1$ in $\Om^{\star\star}$ and $|\nabla F| = 2^{-1} c$ on $\p \Om^{\star\star}$.  By Theorem D in Section \ref{S:a} we infer that there exist $r>0$ and $\sigma_0\in \R^k$ such that  $\Om^{\star\star}$ is the Euclidean ball
\[
B_{e}(r)(0,\sigma_0) = \{(w,\sigma)\in \R^\frac{m+ p}2\times \R^k\mid |w|^2 + |\sigma-\sigma_0|^2 < r^2\}
\]
centred at the point $(0,\sigma_0)$ with radius $r$, and that from \eqref{fgl} we must have
\begin{equation}\label{fglgl}
G(w,\sigma) = 2^{\frac{p}{p-1}} \frac{p-1}{p n^{\frac{1}{p-1}}} \left(r^{\frac{p}{p-1}} - (|w|^2 + |\sigma-\sigma_0|^2)^{\frac{p}{2(p-1)}}\right).
\end{equation}
If we now choose $R>0$ such that $r = \frac{R^2}{4}$, and we keep in mind that $n = \frac{m+ p}2 + k = \frac{Q+p}2$, where $Q = m+2k$ is the homogeneous dimension of $\bG$, from the substitution $|w| = \xi = \frac{|z|^2}{4}$ and the relation \eqref{fgG}, we finally obtain
\[
f(z,\sigma) = \frac{p-1}{2p(Q+p)^{\frac{1}{p-1}}} \left(R^{\frac{2p}{p-1}} - N(z,\sigma-\sigma_0)^{\frac{2p}{p-1}}\right).
\]
This is exactly the formula \eqref{solball} in Proposition \ref{P:gauge}, after a left-translation by $(0,-\sigma_0)$ along the center of $\bG$.

\end{proof}

We next give the

\begin{proof}[Proof of Proposition \ref{P:iips}]
From the proof of Theorem \ref{T:glps} we know that the function  $F = 2^{-\frac{p}{p-1}} G$ solves $\Delta_p F = -1$ in $\Om^{\star\star}$ and $|\nabla F| = 2^{-1} c$ on $\p \Om^{\star\star}$.
To proceed in the discussion we recall that, in the work \cite{GL}, one crucial step in the proof of Theorem D above was to show that
the function
\begin{equation}\label{gl}
P^{\star\star}(x) = |\nabla G(x)|^p + \frac{p}{n(p-1)} G(x)
\end{equation}
satisfies the integral constraint
\begin{equation}\label{Pgl}
\int_{\Om^{\star\star}} P^{\star\star}(x) dx = a^p \operatorname{Vol}_n(\Om^{\star\star}).
\end{equation}
With the aid of \eqref{Pgl} the authors were able to show that the function $P^{\star\star}(x)$ must satisfy the pointwise constraint $P^{\star\star}(x) \equiv a^p$ in $\Om^{\star\star}$, and from this,  by an elaborate argument inspired to that of Weinberger in \cite{hans}, they proved that $\Om^{\star\star}$ must be a ball and $G$ spherically symmetric about its center. Since we presently have $a = 2c^{-1}$, we must have
in $\Om^{\star\star}$
\[
P^{\star\star}(w,\sigma) = |2^{-\frac{p}{p-1}} \nabla G(w,\sigma)|^p + \frac{p}{n(p-1)} 2^{-\frac{p}{p-1}} G(w,\sigma) \equiv (2^{-1} c)^p.
\]
Keeping in mind that $G(w,\sigma) = g(|w|,\sigma)$,  this identity becomes in terms of the function $g$
\[
|\nabla g(|w|,\sigma)|^p + \frac{p}{n(p-1)} 2^{p} g(|w|,\sigma) \equiv 2^{p} c^p.
\]
By \eqref{gradG}, and keeping \eqref{fgG} and the equations $\xi^{\frac p2} = \frac{|z|^p}{2^p}$ and $n = \frac{Q+p}2$ in mind, we finally obtain in terms of $f$ in $\Om$
\[
|\nh f|^p + \frac{2p}{(p-1)(Q+p)} |z|^p f = c^p |z|^p,
\]
which proves the proposition.

\end{proof}

\medskip

Finally, we present the

\begin{proof}[Proof of Theorem \ref{T:gsps}]
Assume that $\Om\subset \bG$ be a connected bounded open set having partial symmetry. Suppose $1<p<Q$ and that $f$ be a weak solution to \eqref{icap} satisfying the overdetermined condition \eqref{gradconv}. We argue similarly to the proof of Theorem \ref{T:glps} and consider $S\in \mathscr S(\mathfrak h)$, and define $g(z,\sigma) = f(Sz,\sigma)$. Since $\Om$ is invariant under the action of $\mathscr S(\mathfrak h)$, invoking \eqref{commnab} in Lemma \ref{L:gradinv} and \eqref{invplap} in Lemma \ref{L:Hlapinv}, and keeping in mind that $|Sz| = |z|$, we see that $g$ is also a weak solution to \eqref{icap} satisfying the overdetermined condition \eqref{gradconv}. By uniqueness, we must have $g \equiv f$ in $\bG\setminus \Om$, and therefore $f(Sz,\sigma) = f(z,\sigma)$ for every $(z,\sigma)\in \bG\setminus \Om$. By the Property (H) we conclude that $f$ must be of the form \eqref{fgG}. Similarly to what was done  before, we conclude that the function $G$ as in \eqref{gG} solves the problem
in $\R^{\frac{m+p}2}\times\R^k\setminus\Om^{\star\star}$
\begin{equation}\label{garsar}
\begin{cases}
\Delta_p G = 0,
\\
\\
G_{\big|_{\p \Om^{\star\star}}} = 1,\ \ \ G_{\big|_{\infty}} = 0,\ \ \ |\nabla G|_{\big|_{\p \Om^{\star\star}}} = 2c.
\end{cases}
\end{equation}
At this point we would like to apply Theorem E in Section \ref{S:a} to conclude that there exist $\sigma_0\in \R^k$ such that, with $ n = \frac{m+p}2 +k$, we have for $r = \frac{n-p}{2(p-1)c}$ that
\[
\Om^{\star\star} = \{(w,\sigma)\in \Rn\mid |w|^2 + |\sigma-\sigma_0|^2 <r\},
\]
and
\begin{equation}\label{GG}
G(w,\sigma) = \left(\frac{r}{(|w|^2 + |\sigma-\sigma_0|^2)^{\frac 14}}\right)^{\frac{n-p}{p-1}}.
\end{equation}
This implementation is possible if our hypothesis that $1<p<Q = m+2k$ implies that $1<p<n$. But in fact, the inequality $p<n = \frac{m+p}2 +k = \frac{Q+p}2$, is equivalent to $p<Q$, and so we can apply Theorem E in Section \ref{S:a}. Since in \eqref{GG} we have $\xi = |w| = \frac{|z|^2}4$, using elementary calculations and the fact $\frac{n-p}{p-1} = \frac{Q-p}{2(p-1)}$, after taking $R>0$ such that $r = \frac{R^2}4$, we finally conclude that
\[
f(z,\sigma) = \left(\frac{R}{N(z,\sigma-\sigma_0)}\right)^{\frac{Q-p}{p-1}},
\]
and that $\Om$ is the gauge ball $B_R(0,\sigma_0)$.
This completes the proof.

\end{proof}

\medskip

We close this section with a result which has a geometric interest and concerns the \emph{horizontal mean curvature} $\mathscr H$ of a non-characteristic hypersurface. This notion was introduced in \cite[Theorem 9.1
\& Definition 9.8]{DGN} as the trace of the horizontal shape operator. Let $\Om\subset\bG$ be a bounded open set of class $C^1$. With $\nu$ being the (Riemannian) outer unit normal on $\p \Om$, we define
\begin{equation}\label{Nhv}
N_H = \sum_{i=1}^m \langle\nu,X_i\rangle X_i.
\end{equation}
We recall that the \emph{characteristic set} $\Sigma$ of $\Om$ is
\begin{equation}\label{Sigma}
\Sigma = \{x\in \p \Om\mid |N_H(x)|^2 = \sum_{i=1}^m \langle\nu(x),X_i(x)\rangle^2 = 0\}.
\end{equation}
This notion was introduced by Fichera in \cite{Fi1}, \cite{Fi2}. It is well-known by now that the set $\Sigma$ is where the devil hides when it comes to boundary value problems for the relevant PDEs in sub-Riemannian geometry. For results on the size of the characteristic set one should see \cite{De1}, \cite{De2}, \cite{Ba}, \cite{Ma}. Assume now that $\Om$ be a connected bounded open set with partial symmetry according to Definition \ref{D:psym}, and assume that $f$ as in \eqref{fgG} be a local defining function for the $\p \Om$, with $g\in C^2(U^\star)$, where $U^\star\supset \Om^\star$ is an open set. As it was pointed out in \cite{CG}, for a domain with partial symmetry one has $\Sigma = \p \Om \cap \{(0,\sigma)\in \bG\}$.
We will need the following result, which is \cite[Proposition 2.6]{Gmanu}.

\begin{proposition}\label{P:H}
At every point of $\p \Om\setminus \Sigma$ one has in terms of a
local defining function $f$ of $\p \Om$
\begin{align*}
|\nh f|^3 \mathscr H =  |\nh f|^2 \sul f -
\Delta_{H,\infty} f.
\end{align*}
\end{proposition}
The reader should note that Proposition \ref{P:H} represents the limiting case $p\to 1$ of \eqref{plaplap} above. Using Proposition \ref{P:H} we obtain the following interesting result.

\begin{proposition}\label{P:Hps}
Suppose that in a neighborhoud of $x=(z,\sigma)\in \p \Om\setminus \Sigma$ the domain is described by \eqref{fgG}, then we have
\[
\mathscr H(x) = \frac{|z|}2 \left\{\frac{1}{|\nabla g|}\left[g_{\xi\xi} + \frac{m-1}{2\xi} g_\xi + \Delta_\sigma g\right] - \frac{\Delta_\infty g}{|\nabla g|^3}\right\},
\]
where all quantities involving $g$ in the square bracket in the right-hand side are evaluated at the point $(\frac{|z|^2}4,\sigma)$.
\end{proposition}

\begin{proof}
From \eqref{gradG} and \eqref{bg2} in Lemma \ref{L:nice}, with $\xi = \frac{|z|^2}4$ we obtain
\begin{align*}
\xi^{\frac 32} |\nabla g|^3 \mathscr H = \xi^2 |\nabla g|^2 \left(g_{\xi\xi} + \frac{m}{2\xi}\ g_\xi+ \Delta_\sigma  g\right) - \xi^2 |\nabla g|^2\left\{\frac{\Delta_\infty g}{|\nabla g|^2} + \frac{g_{\xi}}{2\xi}\right\},
\end{align*}
which implies the desired conclusion.

\end{proof}

Proposition \ref{P:Hps} has the following notable consequence.

\begin{proposition}\label{P:bubbleps}
Suppose $\Om\subset \bG$ be a $C^2$ domain with partial symmetry, and assume that there exist $\alpha>0$ such that, outside the characteristic set $\Sigma$ one has $\mathscr H(x) = \alpha |z|$. Then, there exist $\sigma_0\in \R^k$ and $R>0$ such that $\Om$ is a gauge ball $B_R(0,\sigma_0)$, i.e.
\[
\Om = \{(z,\sigma)\in \bG\mid |z|^4 + 16|\sigma-\sigma_0|^2 < R^4\}.
\]
\end{proposition}

\begin{proof}
Consider the set $\Om^\star \subset \R\times \R^k$ introduced in Definition \ref{D:psym} and let $(\overline \xi,\overline \sigma)\in \p \Om^\star$ be such that if $(\overline z,\overline \sigma)\in \p \Om\setminus \Sigma$, then $\overline \xi = \frac{|\overline z|^2}4>0$. The function $g(\xi,\sigma)$ is a $C^2$ local defining function of $\p \Om^\star$.
If we now set $n = \frac{m+ 1}2 + k$ (never mind the fact that the dimension $n$ is an integer plus $1/2$), we consider the connected bounded open set $\Om^{\star\star}$ in the space $\Rn$, with variables $(w,\sigma)$, where $w\in \R^\frac{m+ 1}2$ and $\sigma\in \R^k$, defined as follows
\begin{equation}\label{Ommaninano}
\Om^{\star\star} = \{(w,\sigma)\in \R^{\frac{m+1}2}\times\R^k\mid  (|w|,\sigma)\in \Om^\star\},
\end{equation}
where $\Om^\star$ is as in Definition \ref{D:psym}.
In $\Om^{\star\star}$ we consider the local defining function
\begin{equation}\label{gGH}
G(w,\sigma) \overset{def}{=} g(|w|,\sigma).
\end{equation}
Notice that, if we now indicate $\xi = |w|$, then we have
\begin{equation}\label{ND}
\begin{cases}
|\nabla G(\overline w,\overline \sigma)| = |\nabla g(\overline \xi,\overline \sigma)|>0,
\\
\text{and}
\\
\Delta G(\overline w,\overline \sigma) = g_{\xi\xi}(\overline \xi,\overline \sigma) + \frac{m-1}{2\overline \xi} g_\xi(\overline \xi,\overline \sigma) + \Delta_\sigma g(\overline \xi,\overline \sigma).
\end{cases}
\end{equation}
It is a classical fact that the standard mean curvature of $\p \Om^{\star\star}$ at a point $(\overline w,\overline \sigma)$ where $\nabla G\not= 0$ is given by
\[
H = \frac{\Delta G}{|\nabla G|}  -
\frac{\Delta_{\infty} G}{|\nabla G|^3}.
\]
Combining this observation with \eqref{ND} and Proposition \ref{P:Hps}, we finally obtain at $(\overline z,\overline \sigma)$
\[
\mathscr H(\overline z,\overline \sigma) = \frac{|\overline z|}2\ H(\overline w,\overline \sigma).
\]
If now $\mathscr H(\overline z,\overline \sigma) = \alpha |\overline z|$, we conclude that
\[
H(\overline w,\overline \sigma) = 2\alpha.
\]
By the arbitrariness of $(\overline z,\overline \sigma)\in \p \Om\setminus \Sigma$, and the continuity of $H$, we finally see that it must be on $\p \Om^{\star\star}$
\[
H\ \equiv\ 2\alpha.
\]
By Alexandrov's soap bubble theorem, see \cite{A} and also \cite{Re}, we infer that $\Om^{\star\star}$ must be a Euclidean ball $\{(w,\sigma)\in \R^{\frac{m+1}2}\times\R^k\mid |w|^2 + (\sigma-\sigma_0)^2 < r^2\}$, with $r = \frac{n-1}{2\alpha} = \frac{Q-1}{4\alpha}$. Since $\xi = |w|$, we conclude from \eqref{Ommani} that $\Om$ must be a gauge ball $B_R(0,\sigma_0)$ of radius $R>0$, with $r = \frac{R^2}4$.

\end{proof}


\section{Appendix}\label{S:a}

In this section we collect two known results from the works \cite{GL} and \cite{GS} that play a critical role in the proofs of Theorem \ref{T:glps}, Proposition \ref{P:iips} and of Theorem \ref{T:gsps}. The former of these results is Theorem D below, the latter is Theorem E below. Consider the standard $p$-Laplacian whose action on a function $G:\Rn\to \R$ is given by
\begin{equation}\label{plapclassic}
\Delta_p G = \operatorname{div}(|\nabla G|^{p-2} \nabla G) = |\nabla G|^{p-2} \left[\Delta G + (p-2) \frac{\Delta_\infty G}{|\nabla G|^2} \right],\quad\quad\quad\quad 1<p<\infty,
\end{equation}
where the $\infty$-Laplacian is defined by
\begin{equation}\label{infty}
\Delta_\infty G = \frac 12 \sa\nabla(|\nabla G|^2),\nabla G\da.
\end{equation}
We remark that the ``dimensionality" in \eqref{plapclassic} is hidden in the term $\Delta G$ in the right-hand side of \eqref{plapclassic}. This comment has been further elucidated in Remarks \ref{R:lap} and \ref{R:plap} above. As we have mentioned in Section \ref{S:intro}, the critical tools that allow us to connect the sub-Riemannian Theorems \ref{T:glps}, \ref{T:gsps} to the Euclidean Theorems D and E are Proposition \ref{P:pcylG}, Lemmas \ref{L:Hlapinv} and \ref{L:gradinv}. For instance, for every $1<p<\infty$, Proposition \ref{P:pcylG} converts the solution $f$ to \eqref{i3}, \eqref{od} into that of Theorem D, in a space of dimension $n = \frac{m+p}2 + k$. Since, as we have mentioned, in a group of Heisenberg type the complex structure induced by the map $J:\mathfrak v \to \operatorname{End}(\mathfrak h)$ in \eqref{J}  forces $m$ to be even, the number $n$ is always an integer when $p =2$. If $p\not=2$, $n$ could be a fractal dimension, but this is inconsequential for our purposes. What matters here is that $n = \frac{Q+p}2$, where $Q = m+2k$ is the homogeneous dimension of $\bG$.

The following statement is \cite[Theorem 1]{GL}, specialised to the operator \eqref{plapclassic}.

\medskip

\noindent \textbf{Theorem D.}\label{T:gl}
\emph{Let $\Om^{\star\star}\subset \Rn$ be a connected, bounded open set and suppose for fixed $p$, $1<p<\infty$, that $G\in W^{1,p}_0(\Om^{\star\star})$ is a nonnegative weak solution to
\begin{equation}\label{pserrin}
\Delta_p G = \operatorname{div}(|\nabla G|^{p-2} \nabla G) = - 1.
\end{equation}
Suppose for some $a>0$ that $|\nabla G(x)|\to a$, $G(x)\to 0$ as $x\to \p \Om^{\star\star}$ in the following sense: Given $\ve>0$ there exists an open set $O = O(\ve) \supset \p \Om^{\star\star}$ such that
\begin{equation}\label{gradgl}
||\nabla G(x)| - a|<\ve,\ \ \ \ G(x)<\ve,
\end{equation}
for a.e. $x\in \Om^{\star\star}\cap O$ with respect to Lebesgue $n$-measure. Then $\Om^{\star\star}$ is a (Euclidean) ball $B_e(x_0,r)$ and
\begin{equation}\label{fgl}
G(x) = \frac{p-1}{p n^{\frac{1}{p-1}}} \left(r^{\frac{p}{p-1}} - |x-x_0|^{\frac{p}{p-1}}\right).
\end{equation}}

\medskip

We emphasise that in Theorem D no regularity assumption is made on $\Om^{\star\star}$, and that Serrin's overdetermined condition $\frac{\p G}{\p \nu} = c$ on $\p \Om^{\star\star}$ is replaced  by the much weaker measure theoretic assumption \eqref{gradgl}. We also reiterate that in \eqref{fgl} the dimension $n$ explicitly shows up \emph{only} in the constant $n^{\frac{1}{p-1}}$. Also note that when $2<p<\infty$ the function in \eqref{fgl} does not have continuous second partials in $B(x_0,r)$. As it is well-known, the optimal regularity for weak solutions of \eqref{plapclassic} is $C^{1,\alpha}_{loc}$.

\medskip

We next state \cite[Theorem 1.1]{GS}. The reader should keep in mind that, in such result, the overdetermined boundary condition $|\nabla G|=a>0$ on $\partial{\Omega}$ is again assumed in the sense of \eqref{gradgl} above. We also mention that in \cite{GS} such result was proved under the additional assumption that $\Om$ be starlike. This hypothesis was more recently removed by Poggesi in \cite[Theorem 1.1]{Po}. To provide the reader with additional historical background, we recall that, using some of the ideas in \cite{AG}, Reichel was able to adapt in \cite{Rei} the method of moving planes to smooth exterior domains and smooth solutions of equations of $p$-Laplacian type, thus proving the radial symmetry for exterior  problems such as the capacitary one. The proof in \cite{GS} was completely different, and used no a priori smoothness of either the ground domain or the solution. It was based on a priori estimates, blowup arguments and integral identities. It ultimately relied on a special $P$-function and on the Alexandrov's soap bubble theorem.

\medskip

\noindent \textbf{Theorem E.}\label{T:gs}
\emph{Let $\Om^{\star\star}\subset \Rn$ be a connected, bounded open set and suppose that $1<p<n$. Denote by $G$ its classical capacitary potential. Then $|\nabla G|=a>0$ on $\partial{\Omega^{\star\star}}$ if and only if
$\Omega^{\star\star}$ is a (Euclidean) ball $B_{e}(x_0,r)$ with radius $r =\frac{n-p}{(p-1)a}.$
The solution is then spherically symmetric about the center $x_0$ of this ball and it is given by
\begin{equation}\label{gs}
G(x)  = \left(\frac{r}{|x-x_0|}\right)^{\frac{n-p}{p-1}}.
\end{equation}}

\vskip 0.2in


\bibliographystyle{amsplain}

\begin{thebibliography}{10}



\bibitem{AG}
G. Alessandrini \& N. Garofalo, \emph{Symmetry for degenerate parabolic equations}. Arch. Rational Mech. Anal. 108 (1989), no. 2, 161-174.

\bibitem{A}
A. D. Alexandrov, \emph{A characteristic property of the
spheres},  Ann. Mat. Pura Appl. \textbf{58} (1962), 303-354.


\bibitem{AS}
W. Ambrose \& I. M. Singer, \emph{
 A theorem on holonomy}.
Trans. Amer. Math. Soc. 75 (1953), 428-443.

\bibitem{Baez02} J. C. Baez,  \emph{The octonions}. Bull. Amer. Math. Soc. (N.S.) 39 (2002), no. 2, 145-205.

\bibitem{BCCu}
A. Baldi, G. Citti \& G. Cupini, \emph{Schauder estimates at the boundary for sub-laplacians in Carnot groups}.
Calc. Var. Partial Differential Equations 58 (2019), no. 6, Paper No. 204, 43 pp.


\bibitem{Ba}
Z. M. Balogh, \emph{Size of characteristic sets and functions with prescribed gradient}. J. Reine Angew. Math. 564 (2003), 63-83.


\bibitem{BGMcv}
A. Banerjee, N. Garofalo \& I. H. Munive, \emph{Compactness methods for $\G^{1,\alpha}$  boundary Schauder estimates in Carnot groups}.
Calc. Var. Partial Differential Equations 58 (2019), no. 3, Paper No. 97, 29 pp.


\bibitem{BGM}
A. Banerjee, N. Garofalo \& I. H. Munive, \emph{Higher order boundary Schauder estimates in Carnot groups}, ArXiv:2210.12950

\bibitem{Barbano98}
P. E. Barbano,  \emph{Automorphisms and quasi-conformal mappings of Heisenberg-type groups.} J. Lie Theory 8 (1998), no. 2, 255--277.

\bibitem{Ber}
M. Berger, \emph{Sur les groupes d'holonomie homog\`ene des vari\'et\'es \`a connexion affine et des vari\'et\'es riemanniennes}.
Bull. Soc. Math. France 83 (1955), 279–330.

\bibitem{Be}
A. L. Besse, \emph{Einstein manifolds}. Reprint of the 1987 edition. Classics in Mathematics. Springer-Verlag, Berlin, 2008. xii+516 pp.




\bibitem{BCZ}
M. Bonk, L. Capogna \& X. Zhou, \emph{Green function in metric measure spaces}, 2022 Preprint. ArXiv:2211.11974

\bibitem{BFM}
T. P. Branson, L. Fontana \& C. Morpurgo, \emph{Moser-Trudinger and Beckner-Onofri's inequalities on the $\operatorname{CR}$ sphere}. Ann. of Math. (2) \textbf{177}~(2013), no. 1, 1-52.

\bibitem{Bryant82}
R. L. Bryant,  \emph{Submanifolds and special structures on the octonians}. J. Differential Geometry 17 (1982), no. 2, 185-232.



\bibitem{CS}
L. Caffarelli \& L. Silvestre, \emph{An extension problem related to the fractional Laplacian}. Comm. Partial Differential Equations 32 (2007), no. 7-9, 124-1260.

\bibitem{CaChMa09} 
O. Calin, D.-C. Chang \& I. Markina, \emph{Geometric analysis on $H$-type groups related to division algebras}. Math. Nachr., 282~ (2009), 44--68.


\bibitem{CCGme}
L. Capogna, G. Citti \& N. Garofalo, \emph{Regularity for a class of quasilinear degenerate parabolic equations in the Heisenberg group}.
Math. Eng. 3 (2021), no. 1, Paper No. 008, 31 pp.

\bibitem{CCZ}
L. Capogna, G. Citti \& X. Zhong, \emph{Lipschitz regularity for solutions of the parabolic $p$-Laplacian in the Heisenberg group}.
Ann. Fenn. Math. 48 (2023), no. 2, 411-428.



\bibitem{CDG}
L. Capogna, D. Danielli \& N. Garofalo, \emph{An embedding theorem and the Harnack inequality for nonlinear subelliptic equations}. Comm. Partial Differential Equations 18 (1993), no. 9-10, 1765-1794.

\bibitem{CDGcap}
L. Capogna, D. Danielli \& N. Garofalo, \emph{Capacitary estimates and the local behavior of solutions of nonlinear subelliptic equations}. Amer. J. Math. 118 (1996), no. 6, 1153-1196.

\bibitem{CG}
L. Capogna \& N. Garofalo, \emph{Boundary behavior of nonnegative solutions of subelliptic equations in NTA domains for Carnot-Carath\'eodory metrics}. J. Fourier Anal. Appl. 4 (1998), no. 4-5, 403-432.

\bibitem{CGN1}
L. Capogna, N. Garofalo \& D. M. Nhieu, \emph{Properties of harmonic measures in the Dirichlet problem for nilpotent Lie groups of Heisenberg type}. Amer. J. Math. 124 (2002), no. 2, 273-306.

\bibitem{CGN2}
L. Capogna, N. Garofalo \& D. M. Nhieu, \emph{Mutual absolute continuity of harmonic and surface measures for H\"ormander type operators}. Perspectives in partial differential equations, harmonic analysis and applications, 49-100, Proc. Sympos. Pure Math., 79, Amer. Math. Soc., Providence, RI, 2008.

\bibitem{cittimu}
G. Citti \& S. Mukherjee, \emph{Regularity of quasi-linear equations with H\"ormander vector fields of step two}.
Adv. Math. 408 (2022), Paper No. 108593, 66 pp.

\bibitem{CoSm} J. H. Conway and D. A. Smith. On Quaternions and Octonions: Their Geometry, Arithmetic, and Symmetry. A K Peters, Ltd., Natick, MA, 2003

\bibitem{CDKR}
M. Cowling, A. H. Dooley, A. Kor\'anyi \& F. Ricci, \emph{$H$-type groups and Iwasawa decompositions}. Adv. Math. 87~(1991), no. 1, 1-41.

\bibitem{CDKR2}
M. Cowling, A. H. Dooley, A. Kor\'anyi \& F.  Ricci,
 \emph{An approach to symmetric spaces of rank one via groups of Heisenberg type.}
 J. of Geom. Anal., \textbf{8}~(1998), no. 2,  199--237.

\bibitem{D}
D. Danielli, \emph{Regularity at the boundary for solutions of nonlinear subelliptic equations}. Indiana Univ. Math. J. 44 (1995), no. 1, 269-286.

\bibitem{DGtrieste}
D. Danielli \& N. Garofalo, \emph{Geometric properties of solutions to subelliptic equations in nilpotent Lie groups}. Reaction diffusion systems (Trieste, 1995), 89–105, Lecture Notes in Pure and Appl. Math., 194, Dekker, New York, 1998.

\bibitem{DG}
D. Danielli \& N. Garofalo, \emph{Green functions in Carnot groups and the geometry of their level sets}. 2000 Preprint.






\bibitem{DGN}
D. Danielli, N. Garofalo \& D. M. Nhieu, \emph{Sub-Riemannian calculus on hypersurfaces in Carnot groups}. Adv. Math. 215 (2007), no. 1, 292-378.

\bibitem{De1}
M. Derridj, \emph{Un probl\'eme aux limites pour une classe d'op\'erateurs du second ordre hypoelliptiques}. Ann. Inst. Fourier Grenoble 21 (1971), 99-148.

\bibitem{De2}
M. Derridj, \emph{Sur un th\'eor\`eme de traces}. Ann. Inst. Fourier Grenoble 22 (1972), 73--83.

\bibitem{El18}  N. Eldredge, \emph{On complex H-type Lie algebras}. Matematiche (Catania) 73 (2018), no. 1, 155-160.

\bibitem{Fi1}
G. Fichera, \emph{Sulle equazioni differenziali lineari ellittico-paraboliche del secondo ordine}. (Italian) Atti Accad. Naz. Lincei Mem. Cl. Sci. Fis. Mat. Natur. Sez. Ia (8) 5 (1956), 1-30.

\bibitem{Fi2}
G. Fichera, \emph{On a unified theory of boundary value problems for elliptic-parabolic equations of second order}. 1960 Boundary problems in differential equations pp. 97-120 Univ. Wisconsin Press, Madison, Wis.


\bibitem{Fo}
G. Folland, \emph{A fundamental solution for a subelliptic operator}, Bull. Amer. Math. Soc., \textbf{79}~(1973), 373-376.

\bibitem{Fo2}
G. B. Folland, \emph{Subelliptic estimates and function spaces on nilpotent Lie groups}. Ark. Mat. 13~(1975), no. 2, 161-207.

\bibitem{FScpam}
G. B. Folland \& E. M. Stein, \emph{Estimates for the $\bar \p_b$ complex and analysis on the Heisenberg group}. Comm. Pure Appl. Math. 27~(1974), 429-522.

\bibitem{FGMT}
R. L. Frank, M. del Mar Gonz\'alez, D. Monticelli \& J. Tan, \emph{An extension problem for the $CR$ fractional Laplacian}. Adv. Math. 270~(2015), 97-137.

\bibitem{Gmanu}
N. Garofalo, \emph{Geometric second derivative estimates in Carnot groups and convexity}. Manuscripta Math. \textbf{126}~(2008), no. 3, 353-373.

\bibitem{Gmanu2}
N. Garofalo, \emph{ Gradient bounds for the horizontal horizontal $p$-Laplacian on a Carnot group and some applications}. Manuscripta Math. \textbf{130}~(2009), no. 3, 375-385.

\bibitem{Gparis}
N. Garofalo, \emph{Hypoelliptic operators and some aspects of analysis and geometry of sub-Riemannian spaces}, Geometry, analysis and dynamics on sub-Riemannian manifolds. Vol. 1, 123-257, EMS Ser. Lect. Math., Eur. Math. Soc., Z\"urich, 2016.

\bibitem{GL}
N. Garofalo \& J. L. Lewis, \emph{
A symmetry result related to some overdetermined boundary value problems}.
Amer. J. Math. 111 (1989), no. 1, 9-33.

\bibitem{GP}
N. Garofalo \& N. C. Phuc, \emph{Boundary behavior of $p$-harmonic functions in the Heisenberg group}. Math. Ann. 351 (2011), no. 3, 587-632.


\bibitem{GR}
N. Garofalo \& K. Rotz, \emph{Properties of a frequency of Almgren type for harmonic functions in Carnot groups}.
Calc. Var. Partial Differential Equations 54~(2015), no. 2, 2197-2238.

\bibitem{GS}
N. Garofalo \& E. Sartori, \emph{Symmetry in exterior boundary value problems for quasilinear elliptic equations via blow-up and a priori estimates}. Adv. Differential Equations 4 (1999), no. 2, 137-161.

\bibitem{GT}
N. Garofalo \& G. Tralli, \emph{Feeling the heat in a group of Heisenberg type}. Adv. Math. 381 (2021), Paper No. 107635, 42 pp.

\bibitem{GTjam}
N. Garofalo \& G. Tralli, \emph{A heat equation approach to intertwining}. J. Anal. Math. 149 (2023), no. 1, 113–134.

\bibitem{GTpot}
N. Garofalo \& G. Tralli, \emph{Heat kernels for a class of hybrid evolution equations}. Potential Analysis, to appear.

\bibitem{GV}
N. Garofalo \& D. Vassilev, \emph{Regularity near the characteristic set in the non-linear Dirichlet problem and conformal geometry of sub-Laplacians on Carnot groups}, Math. Ann. 318~(2000), no. 3, 453-516.

\bibitem{GVduke}
N. Garofalo \& D. Vassilev, \emph{Symmetry properties of positive entire solutions of Yamabe-type equations on groups of Heisenberg type}. Duke Math. J. 106 (2001), no. 3, 411-448.

\bibitem{GMT}
C. Guidi, V. Martino \& G. Tralli, \emph{A characterization of gauge balls in $\Hn$ by horizontal curvatures}, 2022, preprint. ArXiv:2207.02181

\bibitem{HH}
J. Heinonen \& I. Holopainen, \emph{Quasiregular maps on Carnot groups}. J. Geom. Anal. 7 (1997), no. 1, 109-148.

\bibitem{Hoch52}
G. Hochschild, \emph{The Automorphism Group of a Lie Group}. Trans. of the Amer. Math. Soc. 72(2)~(1952), 209--216.

\bibitem{Ho}
L. H{\"o}rmander, \emph{Hypoelliptic second order differential equations}, Acta Math. 119~(1967), 147-171.




\bibitem{IV2}   
S. Ivanov \& D. Vassilev, \emph{Extremals for the Sobolev Inequality and the Quaternionic Contact
Yamabe Problem}, World Scientific Publishing Co. Pte. Ltd., Hackensack, NJ, 2011.




\bibitem{IV15} 
S. Ivanov \& D. Vassilev,  \emph{The Lichnerowicz and Obata first eigenvalue theorems and the Obata uniqueness result in the Yamabe problem on CR and quaternionic contact manifolds.} Nonlinear Anal. 126 (2015), 262--323.


\bibitem{Ka}
A. Kaplan, \emph{Fundamental solutions for a class of hypoelliptic PDE generated
by composition of quadratic forms}, Trans. Amer. Math. Soc., 258, 1 ~(1980), 147-153.

\bibitem{Ka2}
A. Kaplan, \emph{Riemannian nilmanifolds attached to Clifford modules}. Geom. Dedicata 11 (1981), no. 2, 127-136.

\bibitem{Ka83}
A. Kaplan, \emph{On the geometry of groups of Heisenberg type}. Bull. London Math. Soc. 15 (1983), no. 1, 35-42.

\bibitem{K84} A. Kaplan, \emph{Composition of quadratic forms in geometry and analysis: some recent applications.} Quadratic and Hermitian forms (Hamilton, Ont., 1983), 193--201, CMS Conf. Proc., 4, Amer. Math. Soc., Providence, RI, 1984.

\bibitem{K84a}  A. Kaplan, \emph{Lie groups of Heisenberg type.} Conference on differential geometry on homogeneous spaces (Turin, 1983). Rend. Sem. Mat. Univ. Politec. Torino 1983, Special Issue, 117--130 (1984).


\bibitem{KP}
A. Kaplan \& R. Putz, \emph{Boundary behavior of harmonic forms on a rank one symmetric space}. Trans. Amer. Math. Soc. \textbf{231}~(1977), no. 2, 369-384.

\bibitem{KapR}
A. Kaplan, \& F. Ricci,  \emph{Harmonic analysis on groups of Heisenberg type.} Harmonic analysis (Cortona, 1982), 416–435,
Lecture Notes in Math., 992, Springer, Berlin, 1983.

\bibitem{Adam85}
A. Kor\'anyi, \emph{Geometric properties of Heisenberg-type groups}. Adv. in Math. 56 (1985), no. 1, 28-38.

\bibitem{KoR}
A. Kor\'anyi \& H. M. Reimann, \emph{Quasiconformal mappings on the Heisenberg group}. Invent. Math. 80 (1985), no. 2, 309-338.

\bibitem{KoRaim}
A. Kor\'anyi \& H. M. Reimann, \emph{Foundations for the theory of quasiconformal mappings on the Heisenberg group}. Adv. Math. 111 (1995), no. 1, 1--87.


\bibitem{KR07} A. Kor\'anyi, \& F. Ricci, \emph{A classification-free construction of rank-one symmetric spaces.} Bull. Kerala Math. Assoc. 2005, Special Issue, 73–-88 (2007)

\bibitem{KR10} A. Kor\'anyi, \& F. Ricci,  \emph{A unified approach to compact symmetric spaces of rank one}. Colloq. Math. 118 (2010), no. 1, 43--87.


\bibitem{Ma}
V. Magnani, \emph{Characteristic points, rectifiability and perimeter measure on stratified groups}. J. Eur. Math. Soc. (JEMS) 8 (2006), no. 4, 585-609.


\bibitem{MaSc93}
 C. A.  Manogue \& J. Schray,  \emph{Finite Lorentz transformations, automorphisms, and division algebras}. J. Math. Phys. 1 August 1993; 34 (8): 3746--3767.


 \bibitem{MM}
 G. A. Margulis  \& G. D. Mostow,
{\em The differential of a quasiconformal mapping on
a Carnot-Carath\'eodory space,\/} Geom. and Func. Anal., 5, no. 2 (1995), 402--433.

\bibitem{M}
V . G. Maz'ja, \emph{Sobolev Spaces}, Springer-Verlag, Berlin, Heidelberg, New York, 1985.


\bibitem{met1}
G. M\'etivier, \emph{Hypoellipticit\'e analytique sur des groupes nilpotents de rang $2$}. (French) Duke Math. J. 47 (1980), no. 1, 195-221.

\bibitem{met2}
G. M\'etivier, \emph{Analytic hypoellipticity for operators with multiple characteristics}. Comm. Partial Differential Equations 6 (1981), no. 1, 1-90.

\bibitem{MS}
D. Montgomery \& H. Samelson, \emph{Transformation groups of spheres}. Ann. of Math. (2) 44 (1943), 454-470.

\bibitem{Mo}
G. D. Mostow, {\em Strong Rigidity of Locally Symmetric Spaces},
Princeton University Press, Princeton, New Jersey (1973).



\bibitem{muzh}
S. Mukherjee \& X. Zhong, \emph{$C^{1,\alpha}$-regularity for variational problems in the Heisenberg group}.
Anal. PDE 14 (2021), no. 2, 567-594.

\bibitem{NSW}
A. Nagel, E. M. Stein \& S. Wainger, \emph{Balls and metrics defined by vector fields. I. Basic properties}, Acta Math. \textbf{155}~(1985), no. 1-2, 103-147.



\bibitem{Ny}
K. Nystr\"om, \emph{$p$-harmonic functions in the Heisenberg group: boundary behaviour in domains well-approximated by non-characteristic hyperplanes}. Math. Ann. 357 (2013), no. 1, 307-353.


\bibitem{Pansu89}
P. Pansu, \emph{M\'etriques de Carnot-Carath\'eodory et
quasiisom\'etries des espaces sym\'etriques de rang un}, Ann. of Math. (2)
129 (1989), no. 1, 1--60.


\bibitem{Po}
G. Poggesi, \emph{Radial symmetry for $p$-harmonic functions in exterior and punctured domains}. Appl. Anal. 98 (2019), no. 10, 1785-1798.


\bibitem{Rei}
W. Reichel, \emph{Radial symmetry for an electrostatic, a capillarity and some fully nonlinear overdetermined problems on exterior domains}. Z. Anal. Anwendungen 15 (1996), no. 3, 619-635.

\bibitem{Re}
R. C. Reilly, \emph{Mean curvature, the Laplacian, and soap
bubbles}, The American Mathematical Monthly, 89, no.3,
~(1982), 180-188.

\bibitem{Rei01}  
H. M. Reimann, \emph{H-type groups and Clifford modules.} Adv. Appl. Clifford Algebras 11 (2001), no. S2, 277-287. 


\bibitem{Riehm82}
C.  Riehm,  \emph{The automorphism group of a composition of quadratic forms.} Trans. Amer. Math. Soc. 269 (1982), no. 2, 403--414.

\bibitem{RT}
L. Roncal \& S. Thangavelu, \emph{Hardy's inequality for fractional powers of the sublaplacian on the Heisenberg group}. Adv. Math. 302~(2016), 106-158.

\bibitem{RT2}
L. Roncal \& S. Thangavelu, \emph{An extension problem and trace Hardy inequality for the sublaplacian on $H$-type groups}. Int. Math. Res. Not. IMRN (2020), no. 14, 4238-4294.


\bibitem{Saal96}
L.  Saal,  \emph{The automorphism group of a Lie algebra of Heisenberg type}. Rend. Sem. Mat. Univ. Politec. Torino 54 (1996), no. 2, 101--113.

\bibitem{Saal09}
L. Saal,  \emph{Gelfand pairs related to groups of Heisenberg type}. Rev. Un. Mat. Argentina 50 (2009), no. 2, 63--74.




\bibitem{se}
J. Serrin, \emph{A symmetry problem in potential theory}. Arch. Rational Mech. Anal. 43 (1971), 304-318.

\bibitem{Sim}
J. Simons, \emph{On the transitivity of holonomy systems}. Ann. of Math. (2) 76 (1962), 213-234.

\bibitem{V}
V. S. Varadarajan, \emph{Lie Groups, Lie Algebras, and Their Representations}, Springer-Verlag,
New York, Berlin, Heidelberg, Tokyo, 1974.

\bibitem{Wal}
W. Walter, \emph{A new approach to minimum and comparison principles for nonlinear ordinary differential operators of second order}. Nonlinear Anal. 25 (1995), no. 9-10, 1071-1078.

\bibitem{WW14}
H. Wang \& W. Wang, \emph{On octonionic regular functions and the Szeg\"o projection on the octonionic Heisenberg group}. Complex Anal. Oper. Theory 8 (2014), no. 6, 1285–-1324.

\bibitem{hans}
H. F. Weinberger, \emph{Remark on the preceding paper of Serrin}. Arch. Rational Mech. Anal. 43 (1971), 319-320.

\bibitem{zhong}
X. Zhong, \emph{Regularity for variational problems in the Heisenberg group}, https://arxiv.org/abs/1711.03284.



\end{thebibliography}

\end{document}